\documentclass[11pt,reqno]{amsart}
\usepackage[T1]{fontenc}
\usepackage{amsmath}
\usepackage{amsfonts}
\usepackage{amssymb}
\usepackage{graphicx}
\usepackage{cite}
\usepackage{color}
\definecolor{MyLinkColor}{rgb}{0,0,0.4}
\newcommand{\R}{{\mathbb R}}
\newcommand{\E}{{\mathcal E}}

\newcommand{\cK}{\mathcal{K}}

\newcommand{\cF}{\mathcal{F}}

\newcommand{\p}{\partial}

\newcommand{\e}{\varepsilon}

\newcommand{\id}{\mathop{\rm id}\nolimits}

\newtheorem{thm}{Theorem}[section]

\newtheorem{lemma}[thm]{Lemma}
\newtheorem{cor}[thm]{Corollary}
\theoremstyle{remark} 

\numberwithin{equation}{section}
\title[]{A gradient flow approach to a thin film approximation\\ of the Muskat problem}

\thanks{ Partially supported by the french-german PROCOPE program 20190SE}

\author[Ph. Lauren\c cot]{Philippe Lauren\c cot}
\address{Institut  de Math\'ematiques de Toulouse, CNRS UMR 5219, Universit\'e de Toulouse, F-31062  Toulouse cedex 9, France}
\email{laurenco@math.univ-toulouse.fr}

\author[B.--V. Matioc]{Bogdan--Vasile Matioc}
\address{Institut f\" ur Mathematik, Universit\" at Wien, Nordbergstra{\ss}e 15, 1090 Wien, {\"O}sterreich}
\email{bogdan-vasile.matioc@univie.ac.at}

\date{\today}


\usepackage[colorlinks=true,linkcolor=MyLinkColor,citecolor=MyLinkColor]{hyperref} %dvips

\begin{document}

%%%%%%%%%%%%%%%%%%%%%%%%%%%%%%%%%%%%%%%%%%%%%%%%%%%%%%%%%%%%%%%%%%%%%%%%%%%
\begin{abstract}
A  fully coupled system of two second-order parabolic degenerate  equations 
arising as a thin film approximation to the Muskat problem is interpreted as a gradient 
flow for the $2$-Wasserstein distance in the space of probability measures with finite second moment.
A variational scheme is then set up and is the starting point of the construction of weak solutions. 
The availability of two Liapunov functionals turns out to be a central tool to obtain the needed regularity 
to identify the Euler-Lagrange equation in the variational scheme. 
\end{abstract}
%%%%%%%%%%%%%%%%%%%%%%%%%%%%%%%%%%%%%%%%%%%%%%%%%%%%%%%%%%%%%%%%%%%%%%%%%%%

\maketitle

%%%%%%%%%%%%%%%%%%%%%%%%%%%%%%%%%%%%%%%%%%%%%%%%%%%%%%%%%%%%%%%%%%%%%%%%%%%
%%%%%%%%%%%%%%%%%%%%%%%%%%%%%%%%%%%%%%%%%%%%%%%%%%%%%%%%%%%%%%%%%%%%%%%%%%%
\section{Introduction}\label{s:intro}
%%%%%%%%%%%%%%%%%%%%%%%%%%%%%%%%%%%%%%%%%%%%%%%%%%%%%%%%%%%%%%%%%%%%%%%%%%%
%%%%%%%%%%%%%%%%%%%%%%%%%%%%%%%%%%%%%%%%%%%%%%%%%%%%%%%%%%%%%%%%%%%%%%%%%%%

The Muskat model is a free boundary problem describing the motion of two immiscible 
fluids with different densities and viscosities in a porous medium (such as intrusion of water into oil). 
Assuming that the thickness of the two fluid layers is small, a thin film approximation to the Muskat problem has been recently 
derived in \cite{EMMxx} for the space and time evolution of the thickness $f=f(t,x)\ge 0$ and $g=g(t,x)\ge 0$ of the two fluids ($f+g$ 
being then the total height of the layer) and reads
\begin{subequations} \label{eq:problem}
\begin{equation}\label{eq:S2}
\left\{
\begin{array}{rcl}
\p_t f & = & (1+R)\p_x\left(   f\p_xf\right)+R\p_x\left(f\p_xg\right),\\[1ex]
\p_t g & = & R_\mu\p_x\left(   g\p_xf\right)+R_\mu \p_x\left(g\p_xg\right),
\end{array}
\right.
\qquad (t,x)\in (0,\infty)\times \R,
\end{equation}
supplemented with the initial conditions
\begin{equation}\label{eq:bc1}
f(0)=f_0,\qquad g(0)=g_0, \quad x\in \R.
\end{equation}
\end{subequations}
Here, $R$ and $R_\mu$ are two positive real numbers depending on the densities and the viscosities of the fluids.
 Since $f$ and $g$ may vanish, \eqref{eq:S2} is a strongly coupled degenerate parabolic system with a
 full diffusion matrix due to the terms $\p_x (f \p_x g)$ and $\p_x (g \p_x f)$. 
There is however an underlying structure which results in the availability of an energy functional 
\begin{equation}
\E(f,g):=\frac{1}{2}\int_\R \left[ f^2+R(f+g)^2 \right]\, dx, \label{eq:sup1}
\end{equation}
which decreases along the flow. More precisely, a formal computation reveals that 
\begin{equation}
\frac{d}{dt}\E(f,g) = - \int_\R \left[ f\ \left( (1+R)\p_x f + R\p_x g \right)^2 + R R_\mu\ g\ \left( \p_x f + \p_x g \right)^2 \right]\, dx\,.\label{dissE2}
\end{equation}
A similar property is actually valid  when \eqref{eq:S2} is set on a bounded interval $(0,L)$ with homogeneous Neumann boundary conditions: in that setting, the stationary solutions are constants and the principle of linearized stability is used in \cite{EMMxx} to construct global classical solutions which stay in a small neighbourhood of positive constant stationary states. Local existence and uniqueness of classical solutions (with positive components) are also established in \cite{EMMxx} by using the general theory for nonlinear parabolic systems developed in \cite{Am93}. Weak solutions have been subsequently constructed in \cite{ELM11} by a compactness method: the first step is to study a regularized system in which the cross-diffusion terms are ``weakened'' and to show that it has global strong solutions, the proof combining the theory from \cite{Am93} for the local well-posedness and suitable estimates for the global existence. Some of these estimates turn out to be independent of the regularisation parameter and provide sufficient information to pass to the limit as the regularisation parameter goes to zero and obtain a weak solution to \eqref{eq:S2} in a second step. A key argument in the analysis of \cite{ELM11} was to notice that there is another Liapunov functional for 
\eqref{eq:S2} given by
\begin{equation}
\mathcal{H}(f,g) := \int_\R \left[ f\ln{f} + \frac{R}{R_\mu}\ g \ln{g} \right]\, dx\,, \label{eq:sup1b}
\end{equation}
which evolves along the flow as follows:
$$
\frac{d}{dt} \mathcal{H}(f,g) = - \int_\R \left[ |\p_x f|^2 + R\ |\p_x f + \p_x g|^2 \right]\,dx\,.
$$
The basic idea behind the above computation is to notice that an alternative formulation of \eqref{eq:S2} is
\[
\left\{
\begin{array}{rcl}
\p_t f & = & \p_x\left[ f\ \p_x \left( (1+R) f + R g \right) \right], \\[1ex]
\p_t g & = & R_\mu\ \p_x\left[ g\ \p_x\left( f + g \right) \right],
\end{array} 
\right.
\qquad (t,x)\in (0,\infty)\times \R,
\]
so that it is rather natural to multiply the $f$-equation by $\ln{f}$ and the $g$-equation by $\ln{g}$ and find nice cancellations after integrating by parts. In this note, we go one step further and observe that a concise formulation of \eqref{eq:S2} is actually
\begin{equation}\label{eq:S2w}
\left\{
\begin{array}{rcl}
\p_t f & = & \displaystyle{ \p_x\left[ f\ \p_x\left( \frac{\delta\E}{\delta f}(f,g) \right) \right]},\\[1ex]
 & & \\
\displaystyle{ \frac{R}{R_\mu}\ \p_t g} & = & \displaystyle{ \p_x\left[ g\ \p_x\left( \frac{\delta \E}{\delta g}(f,g) \right) \right]},
\end{array}
\right.
\qquad (t,x)\in (0,\infty)\times \R,
\end{equation}
which is strongly reminiscent of the interpretation of second-order parabolic equations as gradient flows with 
respect to the $2$-Wasserstein distance, see \cite[Chapter~11]{AGS08} and \cite[Chapter~8]{Vi03}.  
Indeed, since the pioneering works \cite{JKO98} on the linear Fokker-Planck equation and \cite{Ot98,Ot01} 
on the porous medium equation, several equations have been interpreted as gradient flows with respect to some 
Wasserstein metrics, including doubly degenerate parabolic equations \cite{Ag05}, a model for type-II semiconductors \cite{AS08}, 
the Smoluchowski-Poisson equation \cite{BCC08}, some kinetic equations \cite{CG04,CMV06}, and some fourth-order degenerate parabolic equations \cite{MMS09}, to give a few examples, see also \cite{AGS08} for a general approach. 
As far as we know,  the system \eqref{eq:S2w} seems to be the first example of a system of parabolic partial differential equations which can be interpreted as a gradient flow for Wasserstein metrics. Let us however mention that the parabolic-parabolic Keller-Segel system arising in the modeling of chemotaxis has a mixed Wasserstein-$L^2$ gradient flow structure \cite{CLxx}. 

The purpose of this note is then to show that the heuristic argument outlined previously can be made rigorous 
and to construct weak solutions to \eqref{eq:problem} by this approach. More precisely, let $\cK$ be the 
convex subset of the Banach space $L^1(\R, (1+x^2) dx) \cap L^2(\R)$ defined by 
\begin{equation}
\cK:=\left\{h\in L^1(\R, (1+x^2) dx) \cap L^2(\R)\,:\, h\ge 0 \text{ a.e. and } \int_\R h(x)\, dx=1\right\}, \label{eq:sup2}
\end{equation}
and consider initial data $(f_0,g_0) \in\cK_2:= \cK\times\cK.$ We next denote the set of Borel probability measures on $\R$ with 
finite second moment by $\mathcal{P}_2(\R)$ and the $2-$Wasserstein distance on $\mathcal{P}_2(\R)$ by $W_2.$ Recall that, given two 
Borel probability measures $\mu$ and $\nu$ in $\mathcal{P}_2(\R),$
\[
W_2^2(\mu,\nu):=\inf_{\pi\in\Pi(\mu,\nu)}\int_{\R^2} |x-y|^2d\pi(x,y)\,,
\]
where $\Pi(\mu,\nu)$  is the set of all probability measures $\pi\in \mathcal{P}(\R^2) $ which have marginals $\mu$ and $\nu $, 
that is $\pi[A\times \R]=\mu[A] $ and $\pi[\R\times B]=\nu[B]$ for all measurable subsets $A$ and $B$ of $\R.$ 
Alternatively, $\pi\in\Pi(\mu,\nu)$ is equivalent to 
$$
\int_{\R^2}(\phi(x)+\psi(y))\, d\pi(x,y)=\int_\R \phi(x)\, d\mu(x)+\int_\R\psi(y)\, d\nu(y) \quad\text{ for all }\quad  (\phi,\psi)\in L^1(\R;\R^2).
$$

With these notation, our result reads:

%%%%%%%%%%%%%%%%%%%%%%%%%%%%%%%%%%%%%%%%%%%%%%%%%%%%%%%%%%%%%%%%%%%%%%%%%%%
\begin{thm}\label{MT:1} Assume that $R>0,$ $R_\mu>0.$  
Given  $\tau>0$  and $(f_0,g_0)\in\cK_2,$ the sequence $(f_\tau^n,g_\tau^n)_{n\ge 0}$ obtained recursively by setting
\begin{eqnarray}
(f_\tau^0,g_\tau^0) & := & (f_0,g_0)\,, \label{eq:sup2b} \\
\cF_\tau^n\left( f_\tau^{n+1},g_\tau^{n+1} \right) & := & \inf_{(u,v)\in\cK_2}{\cF_\tau^n(u,v)}\,, \label{eq:sup2c}
\end{eqnarray}
with
\begin{equation}
\cF_\tau^n(u,v) := \frac{1}{2\tau}\ \left(W_2^2\left( u,f_\tau^n \right) + \frac{R}{R_\mu}\ W_2^2\left( v,g_\tau^n\right) \right) + \E(u,v)\,, \quad (u,v)\in\cK_2\,, \label{eq:sup2d}
\end{equation}
is well-defined. 
Introducing the interpolation $(f_\tau,g_\tau)$ defined by
\begin{equation}\label{eq:interp}
\text{$f_\tau(t):=f_\tau^n$ and $g_\tau(t):=g_\tau^n$ for $t\in [n\tau,(n+1)\tau)$ and $n\ge 0,$} 
\end{equation}
there exist a sequence  $(\tau_k)_{k\ge 1}$ of positive real numbers, $\tau_k\searrow 0$, and functions  $(f,g):[0,\infty)\to  \cK_2$ such that
\begin{equation}\label{T1} 
\text{$(f_{\tau_k}, g_{\tau_k})\longrightarrow (f,g)$ in $L^2((0,T)\times\R;\R^2)$ for all $T>0.$ }
\end{equation}
Moreover, 
\begin{itemize}
\item[$(i)$] $(f,g)\in L^\infty(0,\infty; L^2(\R;\R^2))$, $(\p_x f,\p_x g) \in L^2(0,t;H^1(\R;\R^2))$,
\item[$(ii)$] $(f,g)\in C ([0,\infty);H^{-3}(\R;\R^2))$ with $(f,g)(0)=(f_0,g_0),$
\end{itemize}
and the pair $(f,g)$ is a weak solution of \eqref{eq:problem} in the sense that 
\begin{equation}\label{T2}   
\left\{
\begin{aligned}
&\int_\R f(t)\ \xi\, dx-\int_\R f_0\ \xi\, dx + \int_0^t\int_\R f(\sigma) \left[ (1+R)\p_xf+R\p_xg\right](\sigma) \p_x\xi\, dx\, d\sigma=0\,,\\[1ex]
&\int_\R g(t)\ \xi\, dx - \int_\R g_0\ \xi\, dx+ R_\mu\ \int_0^t \int_\R g(\sigma) \left( \p_xf+ \p_xg \right)(\sigma)\p_x\xi\, dx\, d\sigma=0\,,
\end{aligned}\right.
\end{equation}
for all $\xi\in C_0^\infty(\R) $ and $t\ge 0.$
 In addition, $(f,g)$ satisfy the following estimates
\begin{align*}
(a) \quad & \mathcal{H}(f(T),g(T)) +\int_0^T\int_\R\left[  |\p_xf|^2 + R |\p_x(f+g)|^2 \right]\, dx\, dt\leq  \mathcal{H}(f_0,g_0)\,,\\[1ex]
(b)\quad &\E(f(T), g(T))
+ \frac{1}{2}\int_{0}^T\int_\R\left[f\left((1+R)\p_xf+R\p_x g\right)^2+RR_\mu g(\p_xf+\p_x g)^2\right]\, dx\,dt \leq \E(f_{0},g_{0})\,,
\end{align*}
 for a.e. $T\in(0,\infty)$, $\E$ and $\mathcal{H}$ being the functionals defined by \eqref{eq:sup1} and \eqref{eq:sup1b}, respectively.
\end{thm}
%%%%%%%%%%%%%%%%%%%%%%%%%%%%%%%%%%%%%%%%%%%%%%%%%%%%%%%%%%%%%%%%%%%%%%%%%%%

Let us briefly outline the proof of Theorem~\ref{MT:1}:  in the next section, we study the variational problem \eqref{eq:sup2c} and the properties of its minimizers. A key argument here is to note that the availability of the Liapunov functional \eqref{eq:sup1b} allows us to apply an argument from \cite{MMS09} which guarantees that the minimizers are not only in $L^2(\R;\R^2)$ but also in $H^1(\R;\R^2)$. This property is crucial in order to derive the Euler-Lagrange equation in Section~\ref{s:ele}. The latter is then used to obtain additional regularity on the minimizers, adapting an argument from \cite{Ot98}. Convergence of the variational approximation is established in Section~\ref{s:conv}. Finally, three technical results are collected in the Appendix. 

\medskip

As a final comment, let us point out that we have assumed for simplicity that the initial data $f_0$ and $g_0$ are probability measures but that the case of initial data having different masses may be handled in the same way after a suitable rescaling: more precisely, let $(f_0,g_0) \in L^2(\R)\cap L^1(\R, (1+x^2) dx)$ and denote a solution to \eqref{eq:problem} by $(f,g)$. Setting $F:=f/\|f_0\|_1$ and $G:=g/\|g_0\|_1$ and recalling that $\|f(t)\|_1=\|f_0\|_1$ and $\|g(t)\|_1=\|g_0\|_1$ for all $t\ge 0$, we realize that $(F,G)$ solves
\[\left\{
\begin{array}{rcl}
\displaystyle\frac{1}{\|g_0\|_1}\ \p_t F & = & \p_x\left[ F\ \p_x \left( (1+R) \eta^2 F + R G \right) \right], \\[1ex]
 & & \\
\displaystyle\frac{R}{R_\mu \|f_0\|_1}\ \p_t G & = & \p_x\left[ G\ \p_x\left( R F + R \eta^{-2} G \right) \right],
\end{array}
\right.
\qquad (t,x)\in (0,\infty)\times \R,
\]
with $\eta^2 := \|f_0\|_1/\|g_0\|_1$ and initial data $(F_0,G_0):=(f_0/\|f_0\|_1,g_0/\|g_0\|_1)\in\cK_2$. 
The corresponding variational scheme then involves the functional  
$$
\frac{1}{2\tau}\ \left( \frac{1}{\|g_0\|_1}\ W_2^2\left( u,F_0 \right) + \frac{R}{R_\mu \|f_0\|_1}\ W_2^2\left( v,G_0 \right) \right) + 
\frac{\eta^2}{2}\ \|u\|_2^2 + \frac{R}{2}\ \left\| \eta\ u + \eta^{-1}\ v \right\|_2^2, \quad (u,v)\in\cK_2\,,
$$
to which the analysis performed below (with $\eta=1$) also applies.

%%%%%%%%%%%%%%%%%%%%%%%%%%%%%%%%%%%%%%%%%%%%%%%%%%%%%%%%%%%%%%%%%%%%%%%%%%%
%%%%%%%%%%%%%%%%%%%%%%%%%%%%%%%%%%%%%%%%%%%%%%%%%%%%%%%%%%%%%%%%%%%%%%%%%%%
\section{A variational scheme}\label{s:varsch}
%%%%%%%%%%%%%%%%%%%%%%%%%%%%%%%%%%%%%%%%%%%%%%%%%%%%%%%%%%%%%%%%%%%%%%%%%%%
%%%%%%%%%%%%%%%%%%%%%%%%%%%%%%%%%%%%%%%%%%%%%%%%%%%%%%%%%%%%%%%%%%%%%%%%%%%

Given  $\tau>0$  and $(f_0,g_0)\in\cK_2,$ we introduce the  functional
\begin{equation}\label{functional}
\cF_\tau(u,v) := \frac{1}{2\tau}\ \left(W_2^2(u,f_0)+\frac{R}{R_\mu}\ W_2^2(v,g_0) \right) + \E(u,v)\,, \quad (u,v)\in\cK_2\,,
\end{equation}
and consider the minimization problem 
\begin{equation}\label{min}
\inf_{(u,v)\in\cK_2}\cF_\tau(u,v).
\end{equation}

%%%%%%%%%%%%%%%%%%%%%%%%%%%%%%%%%%%%%%%%%%%%%%%%%%%%%%%%%%%%%%%%%%%%%%%%%%%
\subsection{Existence and properties of minimizers}\label{s:epm}
%%%%%%%%%%%%%%%%%%%%%%%%%%%%%%%%%%%%%%%%%%%%%%%%%%%%%%%%%%%%%%%%%%%%%%%%%%%

Let us start by proving that, for each $(f_0,g_0)\in\cK_2$, the minimization problem \eqref{min} has a unique solution in $\cK_2.$

%%%%%%%%%%%%%%%%%%%%%%%%%%%%%%%%%%%%%%%%%%%%%%%%%%%%%%%%%%%%%%%%%%%%%%%%%%%
\begin{lemma}\label{Unimin} Given $(f_0,g_0)\in\cK_2$ and $\tau>0,$ there exists a unique minimizer $(f,g)\in\cK_2$ of \eqref{min}.
Additionally, $(f,g)\in H^1(\R;\R^2)$ with
\begin{equation}\label{Heat3}
\|\p_xf\|_2^2+R\|\p_x(f+g)\|_2^2\leq\frac{1}{\tau}\ \left[ H(f_0)-H(f) +\frac{R}{R_\mu}\ \left( H(g_0)-H(g) \right) \right]\,,
\end{equation}
where
\begin{equation}\label{eq:H}
H(h):=\int_\R h\ln (h)\, dx \qquad\text{for $h\in L^1(\R)$  such that $h\ge 0$ a.e. and $h\ln(h)\in L^1(\R).$}
\end{equation}
\end{lemma}
%%%%%%%%%%%%%%%%%%%%%%%%%%%%%%%%%%%%%%%%%%%%%%%%%%%%%%%%%%%%%%%%%%%%%%%%%%%

 Recall that, if $h\in\cK$, then $h\ln{h}\in L^1(\R)$ (see Lemma~\ref{le:ap1} below) so that the right-hand side of \eqref{Heat3} is well-defined.
 
\begin{proof} 
The uniqueness of the minimizer follows from  the convexity of $\cK_2$ and $W_2^2$ and the strict convexity of the energy functional $\E$. 

We next prove the existence of a minimizer. To this end, pick a minimizing sequence $(u_k,v_k)_{k\ge 1}\in \cK_2.$ 
There exists a constant $C>0$  such that 
\begin{eqnarray}
\| u_k\|_2 + \|v_k\|_2 & \le & C\,, \quad k\ge 1\,, \label{eq:sup3} \\
W_2(u_k,f_0) + W_2(v_k,g_0) & \le & C\,, \quad k\ge 1\,. \label{eq:sup4} 
\end{eqnarray}
From \eqref{eq:sup3} we obtain at once that there  exist $(f,g)\in L^2(\R;\R^2)$ and a subsequence of $(u_k,v_k)_{k\ge 1}$ (denoted again by $(u_k,v_k)_{k\ge 1}$) such that
\begin{equation}
u_k\rightharpoonup f\qquad\text{and}\qquad v_k\rightharpoonup g\qquad\text{in $L^2(\R).$} \label{eq:sup5}
\end{equation}
Let us first  check that $(f,g)\in\cK_2.$ 
 Indeed, the nonnegativity of $f$ and $g$ readily follows from that of $u_k$ and $v_k$ by \eqref{eq:sup5} while integrating the inequality $x^2\le 2y^2+2|x-y|^2$ with respect to  an arbitrary $\pi\in\Pi(u_k,f_0)$ yields
\begin{eqnarray*}
 \int_\R u_k(x) x^2\, dx = \int_{\R^2} x^2\, d\pi(x,y) & \le & 2  \int_{\R^2} y^2\, d\pi(x,y)+2\int_{\R^2} |x-y|^2\, d\pi(x,y) \\
& \le &  2 \int_{\R^2} f_0(y) y^2\, dy + 2\int_{\R^2} |x-y|^2\, d\pi(x,y) \,,
\end{eqnarray*}
which implies by virtue of \eqref{eq:sup4} that 
\begin{equation}\label{eq:1}
\int_\R u_k(x) x^2\, dx\leq 2\int_\R f_0(x) x^2\, dx + 2 W_2^2(u_k,f_0) \le C\,, \qquad k\ge 1\,.
\end{equation}
 Similarly,
\begin{equation}\label{eq:1b}
\int_\R v_k(x) x^2\, dx \le C\,, \qquad k\ge 1\,.
\end{equation}
Owing to \eqref{eq:sup3}, \eqref{eq:1}, and \eqref{eq:1b}, we deduce from the Dunford-Pettis theorem that $(u_k)_{k\ge 1}$ and $(v_k)_{k\ge 1}$ are weakly sequentially compact in $L^1(\R).$ We may thus assume (after possibly extracting a further subsequence) that $u_k\rightharpoonup f$ and $v_k\rightharpoonup g$ in $L^1(\R),$ whence  
\[
\int_\R f(x)\, dx = \lim_{k\to\infty} \int_\R u_k(x)\, dx = 1 \quad\text{ and }\quad \int_\R g(x)\, dx = \lim_{k\to\infty} \int_R v_k(x)\, dx=1\,.
\]
Finally, combining \eqref{eq:sup3}, \eqref{eq:1}, and \eqref{eq:1b} with a truncation argument ensure that $f$ and $g$ both belong to $L^1(\R,(1+x^2)dx).$
 Summarising, we have shown that $(f,g)\in \cK_2.$

The next step is to prove that 
\[
\cF_\tau(f,g)=\inf_{(u,v)\in \cK_2} \cF_\tau(u,v) \,.
\]
 Indeed, on the one hand, the  weak convergence \eqref{eq:sup5} implies that
\[
\E(f,g)\leq \liminf_{k\to\infty}\E(u_k,v_k).
\]
 On the other hand, we recall that the $2$-Wasserstein metric $W_2$ is lower semicontinuous with respect to the  narrow convergence of  probability measures in each of its arguments,  
see \cite[Proposition~7.1.3]{AGS08}, and the weak convergence of $(u_k,v_k)_{k\ge 1}$ in $L^1(\R;\R^2)$ ensures that
\[
W_2^2(f,f_0)\leq \liminf_{k\to\infty} W_2^2(u_k,f_0)\qquad\text{and}\qquad W_2^2(g,g_0)\leq \liminf_{k\to\infty} W_2^2(v_k,g_0)\,.
\]
Consequently, 
$$
\cF_\tau(f,g) \leq \liminf_{k\to\infty} \cF_\tau(u_k,v_k) \;\;\text{ with }\;\; (f,g)\in\cK_2\,,
$$
so that $(f,g)$ is a minimizer of $\cF_\tau$ in $\cK_2$.

As a final step, we show that $f$ and $g$ belong to $H^1(\R).$ 
 To this end, we follow the approach developed in \cite{MMS09} and take advantage of the availability of another
 Liapunov function as already discussed in the Introduction. More precisely, denote  the heat semigroup by $(G_t)_{t\ge 0}$, that is,
\[
(G_t h)(x) := \frac{1}{\sqrt{4\pi t}}\ \int_\R \exp{\left( - \frac{|x-y|^2}{4t} \right)}\ h(y)\, dy\,, \quad (t,x)\in [0,\infty)\times\R\,,
\]  
for $h\in L^1(\R)$. 
Since $(f,g)\in\cK_2$, classical properties of the heat semigroup ensure that $(G_t f, G_t g)\in\cK_2$ for all $t\ge 0$. 
Consequently,  $\cF_\tau(f, g)\leq\cF_\tau(G_tf, G_tg)$ and we deduce that
\begin{equation}\label{Heat-1}
\begin{aligned}
\E(f,g) & - \E(G_tf,G_t g) \leq \frac{1}{2\tau}\left[\left(W_2^2(G_tf,f_0)-W_2^2(f,f_0)\right) + \frac{R}{ R_\mu} \left(W_2^2(G_tg,g_0)-W_2^2(g,g_0)\right)\right]
\end{aligned}
\end{equation}
for all $t\geq0.$
Moreover, for all $t>0$, we have 
\begin{align*}
\frac{d}{dt}\E(G_tf,G_tg)=&\int_{\R} \left[ G_tf\ \p_t G_tf + R\ (G_tf+G_tg)\ \p_t(G_tf+G_tg) \right]\, dx=-\|\p_xG_tf\|_2^2-R\ \|\p_x G_t(f+g)\|_2^2,
\end{align*}
and by integration with respect to time we find that
$$
\frac{1}{t}\ \int_0^t \left[ \|\p_x G_s f\|_2^2+R\|\p_x G_s (f+g)\|_2^2 \right] \, ds\leq  \frac{\E(f,g)-\E(G_tf,G_tg)}{t}\qquad\text{for all $t>0.$}
$$
 Since $s\mapsto \|\p_x G_s h\|_2$ is non-increasing  for $h\in L^1(\R)$ we end up with
\begin{equation}\label{Heat0}
\|\p_x G_t f\|_2^2+R\|\p_x G_t (f+g)\|_2^2 \leq  \frac{\E(f,g)-\E(G_tf,G_tg)}{t}\qquad\text{for all $t>0.$}
\end{equation}
We recall now some properties of the heat flow  in connection with the $2$-Wasserstein distance $W_2$, see \cite{AGS08,CMV06,Ot01}, these properties being actually collected in \cite[Theorem~2.4]{MMS09}.
 The heat flow is the gradient flow of the entropy functional $H$ given by \eqref{eq:H} for $W_2$ and, for all  $(h,\tilde{h})\in \cK_2,$ we have \cite[Theorem~11.1.4]{AGS08}
\begin{equation}\label{Heat}
\frac{1}{2}\frac{d}{dt}W_2^2(G_t h,\tilde{h}) + H(G_t h)\leq H(\tilde{h}) \qquad \text{for a.e. $t\geq0$}.
\end{equation}
Choosing  $(h,\tilde{h})=(f,f_0)$ and $(h,\tilde{h})=(g,g_0)$ in \eqref{Heat}, we obtain 
$$
\frac{1}{2}\frac{d}{dt}\left[ W_2^2(G_tf,f_0)+\frac{R}{R_\mu}W_2^2(G_t g,g_0) \right] \leq H(f_0)-H(G_t f) +\frac{R}{R_\mu}\left( H(g_0)-H(G_tg) \right)
$$
for a.e. $t\ge 0.$ Integrating the above inequality with respect to time and using the time monotonicity of $s\mapsto H(G_s f)$ and $s\mapsto H(G_s g)$ give
\begin{eqnarray}
& & \frac{1}{2} \left[ W_2^2(G_tf,f_0) - W_2^2(f,f_0) +\frac{R}{R_\mu}\ \left( W_2^2(G_t g,g_0) - W_2^2(g,g_0) \right) \right] \nonumber \\
& & \leq \int_0^t \left[ H(f_0)-H(G_s f) + \frac{R}{R_\mu}\ \left( H(g_0)-H(G_s g) \right) \right]\, ds \nonumber \\
& & \leq t\ \left[ H(f_0)-H(G_t f) + \frac{R}{R_\mu}\ \left( H(g_0)-H(G_t g) \right) \right]. \label{Heat1}
\end{eqnarray}
Gathering \eqref{Heat-1}, \eqref{Heat0}, and \eqref{Heat1}, we find
\begin{equation}\label{Heat2}
\|\p_x G_t f\|_2^2 + R\ \|\p_x G_t(f+g)\|_2^2 \leq\frac{1}{\tau}\ \left[ H(f_0)-H(G_t f) + \frac{R}{R_\mu}\ \left( H(g_0) - H(G_t g) \right) \right]
\end{equation}
for $t>0.$ 
As a direct consequence of \eqref{Heat2}  and the boundedness from below \eqref{eq:ap2} of $H$ in $\cK$, 
$(\p_x G_t f)_{t>0}$ and $(\p_xG_t g)_{t>0}$ are bounded in $L^2(\R)$ and converge to $\p_x f$ and $\p_x g$, respectively, in the sense of distributions as $t\to 0$. This implies that both $f$ and $g$ belongs to $H^1(\R)$ and we can pass to the limit as $t\to 0$ in \eqref{Heat2} to  obtain the desired estimate \eqref{Heat3} and finish the proof.
\end{proof}

%%%%%%%%%%%%%%%%%%%%%%%%%%%%%%%%%%%%%%%%%%%%%%%%%%%%%%%%%%%%%%%%%%%%%%%%%%%
\subsection{The Euler-Lagrange equation}\label{s:ele}
%%%%%%%%%%%%%%%%%%%%%%%%%%%%%%%%%%%%%%%%%%%%%%%%%%%%%%%%%%%%%%%%%%%%%%%%%%%

 We now identify the Euler-Lagrange equation corresponding to the minimization problem \eqref{min}. 

%%%%%%%%%%%%%%%%%%%%%%%%%%%%%%%%%%%%%%%%%%%%%%%%%%%%%%%%%%%%%%%%%%%%%%%%%%%
\begin{lemma}\label{le:sup1}
Given $(f_0,g_0)\in\cK_2$ and $\tau>0$, the minimizer $(f,g)$ of $\cF_\tau$ in $\cK_2$ satisfies
\begin{eqnarray}
\left| \frac{1}{\tau}\ \int_\R \xi\ (f-f_0)\, dx + \int_\R \left[ \left( (1+R)\ f\ \p_x f + R\ f\ \p_x g \right)\ \p_x\xi \right]\, dx \right| & \le & \frac{\|\p_x^2\xi\|_\infty}{2\tau}\ W_2^2(f,f_0)\,, \label{eq:sup6} \\
\left| \frac{1}{\tau}\ \int_\R \xi\ (g-g_0)\, dx + R_\mu\ \int_\R \left[ \left( g\ \p_x f + g\ \p_x g \right)\ \p_x\xi \right] \, dx \right| & \le & \frac{\|\p_x^2\xi\|_\infty}{2\tau}\ W_2^2(g,g_0)\,, \label{eq:sup7}
\end{eqnarray}
for $\xi\in C_0^\infty(\R)$.
\end{lemma}
%%%%%%%%%%%%%%%%%%%%%%%%%%%%%%%%%%%%%%%%%%%%%%%%%%%%%%%%%%%%%%%%%%%%%%%%%%%

\begin{proof}  To derive \eqref{eq:sup6}-\eqref{eq:sup7} we follow the general strategy outlined in \cite[Chapter~8]{Vi03}. 
According to Brenier's theorem \cite[Theorem 2.12]{Vi03}, there are two convex functions $\varphi:\R\to\R$ and $\psi:\R\to\R$ which are uniquely determined up to an additive constant such that 
\begin{subequations} \label{eq:Brenier}
\begin{equation}\label{Brenier}
W_2^2(f,f_0)=\int_\R |x - \p_x\varphi(x)|^2\ f_0(x)\, dx = \inf_{T\#f_0=f} \int_\R |x-T(x)|^2\ f_0(x)\, dx,
\end{equation}
where the infimum is taken over all measurable functions $T:\R\to\R$ pushing $f_0$ forward to $f$ ($f=T\#f_0$), i.e. satisfying
\[
\int_{B} f(x)\, dx = \int_{T^{-1}(B)} f_0(x)\, dx \quad\text{for all Borel sets $B$ of $\R,$}
\]
and
\begin{equation}\label{BrenierB}
\quad W_2^2(g,g_0) = \int_\R |x-\p_x\psi(x)|^2\ g_0(x)\, dx = \inf_{S\#g_0=g} \int_\R |x-S(x)|^2\ g_0(x)\, dx.
\end{equation}
\end{subequations}
We pick now two test functions $\eta$ and $\xi$ in $C^\infty_0(\R) $ and define 
\begin{equation}\label{eq:app2}
f_\e := ((\id+\e\xi)\circ \p_x\varphi) \# f_0 = (\id+\e\xi) \# f \quad\text{ and }\quad g_\e := ((\id+\e\eta)\circ \p_x\psi) \# g_0 = (\id+\e\eta) \# g
\end{equation}
for each $\e\in[0,1]$, where $\id$ is the identity function on $\R.$ To ease notation we set
\begin{equation}\label{diff}
T_\e := \id+\e\xi \quad\text{ and }\quad S_\e := \id+\e\eta,
\end{equation}
and observe that  there is $\e_0$ small enough (depending on both $\xi$ and $\eta$) such that, for $\e\in[0,\e_0],$ $T_\e$ and $S_\e$ are $C^\infty-$diffeomorphisms in $\R$. 
 Then, by \eqref{eq:app2}, we find the identities
\begin{equation}\label{eq:app} 
f_\e = \frac{f\circ T_\e^{-1}}{\p_x T_\e\circ T_\e^{-1}} \quad\text{ and }\quad g_\e = \frac{g\circ S_\e^{-1}}{\p_x S_\e \circ S_\e^{-1}},  \qquad \e\in (0,\e_0].
\end{equation}
 Observing that $\|f_\e\|_1=\|f\|_1=\|g_\e\|_1=\|g\|_1=1$ and
\begin{equation}\label{2-1}
\|f_\e\|_2^2=\int_\R \frac{|f(x)|^2}{\p_x T_\e(x)}\, dx \qquad\text{and}\qquad \|g_\e\|_2^2 = \int_\R\frac{|g(x)|^2}{\p_x S_\e(x)}\, dx,
\end{equation}
we clearly
 have $(f_\e,g_\e)\in \cK_2$ for all $\e\in(0,\e_0]$  and thus $\cF_\tau(f,g) \le \cF_\tau(f_\e,g_\e)$. 
 Consequently,
\begin{equation}\label{inequ}
0 \le \frac{1}{2\tau}\ \left[ W_2^2(f_\e,f_0) - W_2^2(f,f_0) + \frac{R}{R_\mu}\ \left( W_2^2(g_\e,g_0) - W_2^2(g,g_0) \right) \right] + \E(f_\e,g_\e) - \E(f,g).
\end{equation}
Concerning the energy $\E$, it follows from \eqref{2-1}  that 
\begin{equation}\label{rel}
2(\E(f_\e,g_\e)-\E(f,g))= (1+R)\ I_1^\e + R\ I_2^\e + 2R\ I_3^\e\,,
\end{equation}
with
\begin{eqnarray*}
& & I_1^\e := \int_\R \left( \frac{1}{ \p_x T_\e(x)}-1 \right)\ |f(x)|^2\, dx\,, \qquad I_2^\e := \int_\R \left( \frac{1}{\p_x S_\e(x)}-1 \right)\ |g(x)|^2\, dx\,, \\
& & \hspace{2cm} I_3^\e := \int_\R \left( f_\e\ g_\e - f\ g \right)(x)\, dx\,. 
\end{eqnarray*}
We now consider the  three integrals in the right-hand side of the relation \eqref{rel} separately: since
\[
I_1^\e=-\e\int_\R\frac{\p_x\xi}{1+\e\p_x\xi}\ f^2\, dx \quad\text{and}\quad I_2^\e=-\e\int_\R\frac{\p_x\eta}{1+\e\p_x\eta}\ g^2\, dx\,,
\]
it readily follows from Lebesgue's dominated convergence theorem that
\begin{equation}\label{L-1}
\lim_{\e\to 0} \frac{I^\e_1}{\e} = - \int_\R \p_x\xi\ f^2\, dx \quad\text{and}\quad \lim_{\e\to 0} \frac{I_2^\e}{\e} = - \int_\R \p_x\eta\ g^2\, dx\,.
\end{equation}

We next turn to the term $I_3^\e$ involving both $f$ and $g$ and  split it in two terms $2I_3^\e = I_{31}^\e + I_{32}^\e$ with
$$
I_{31}^\e := \int_{\R} (f_\e+f)\ (g_\e-g)\, dx \quad\text{and}\quad I_{32}^\e := \int_\R (g_\e+g)\ (f_\e-f)\, dx\,.
$$
By \eqref{eq:app},
\begin{align*}
I_{31}^\e=&\int_\R\left(\frac{f\circ T_\e^{-1}}{\p_x T_\e\circ T_\e^{-1}}+f\right)\left(\frac{g\circ S_\e^{-1}}{\p_x S_\e\circ S_\e^{-1}}-g\right)\, dx \\
=&\int_\R\left(\frac{f\circ T_\e^{-1}\circ S_\e}{\p_x T_\e\circ T_\e^{-1}\circ S_\e}+f\circ S_\e\right)\left(g- (g\circ S_\e)\ \p_x S_\e \right)\, dx\\
=&\int_\R\left(\frac{f\circ T_\e^{-1}\circ S_\e}{\p_x T_\e\circ T_\e^{-1}\circ S_\e}+f\circ S_\e\right)\left(g- g\circ S_\e\right)\, dx 
-\e\int_\R \p_x\eta\ (g\circ S_\e)\ \left(\frac{f\circ T_\e^{-1}\circ S_\e}{\p_x T_\e\circ T_\e^{-1}\circ S_\e}+f\circ S_\e\right)\, dx.
\end{align*}
On the one hand, invoking Lemma~\ref{L:1} (with $(h,\zeta)=(g,\eta)$),  we know that $(g-g\circ S_\e)/\e\rightharpoonup - \eta \p_x g$ in $L^2(\R)$ as $\e\to 0$. On the other hand, using again Lemma~\ref{L:1}  as well as Lemma~\ref{le:ap2}, we have that $f\circ S_\e \longrightarrow f$ and $f\circ T_\e^{-1} \circ S_\e \longrightarrow f$ in $L^2(\R)$ as $\e\to 0$, and so does $f\circ T_\e^{-1}\circ S_\e/ (\p_x T_\e\circ T_\e^{-1} \circ S_\e)$ owing to the uniform convergence of $(\p_x T_\e)_\e$ to $1$. Consequently,
\begin{align}\label{L-3}
\lim_{\e\to 0} \frac{I^\e_{31}}{\e} = &-2\int_\R f\ \p_x(\eta g)\, dx,
\end{align}
and similarly
\begin{align}\label{L-4}
\lim_{\e\to 0} \frac{I^\e_{32}}{\e} = &-2\int_\R g\ \p_x(\xi f)\, dx.
\end{align}
Gathering \eqref{L-1}-\eqref{L-4}, it follows from \eqref{rel} that
\begin{equation}\label{Est1}
\lim_{\e\to 0} \frac{\E(f_\e,g_\e)-\E(f,g)}{\e}=-(1+R)\ \int_\R \frac{f^2}{2}\ \p_x\xi\, dx - R\ \int_\R \frac{g^2}{2}\ \p_x\eta\, dx - R\ \int_\R \left[ f\ \p_x(\eta g) + g\ \p_x(\xi f) \right]\, dx. 
\end{equation} 
To handle the terms of \eqref{inequ} involving the Wasserstein distance, we  argue as in \cite[Section~8.4]{Vi03} and write
\begin{align*}
W_2^2(f_\e,f_0)& \leq  \int_\R |\id-T_\e\circ\p_x\varphi|^2\ f_0\, dx = \int_\R |\id-\p_x\varphi - \e\ \xi\circ\p_x\varphi|^2\ f_0\, dx\\
&=\int_\R |\id-\p_x\varphi|^2\ f_0\, dx - 2 \e\ \int_\R (\id-\p_x\varphi)\ (\xi\circ\p_x\varphi)\ f_0\, dx + \e^2\ \int_\R |\xi\circ\p_x\varphi|^2\ f_0\, dx,
\end{align*}
from which we deduce,  according to the definition of $\p_x\varphi$, 
\begin{equation}\label{F-1}
W_2^2(f_\e,f_0) \leq W_2^2(f,f_0) - 2\e\ \int_\R (\id-\p_x\varphi)\ (\xi\circ\p_x\varphi)\ f_0\, dx + \e^2\ \int_\R |\xi\circ\p_x\varphi|^2\ f_0\, dx,
\end{equation}
and similarly
\begin{equation}\label{F-2}
W_2^2(g_\e,g_0) \leq W_2^2(g,g_0) - 2\e\ \int_\R(\id-\p_x\psi)\ (\eta\circ\p_x\psi)\ g_0\, dx + \e^2\ \int_\R |\eta\circ\p_x\psi|^2\ g_0\, dx.
\end{equation}
Summing \eqref{Est1}, \eqref{F-1}, and \eqref{F-2}, we obtain by dividing \eqref{inequ} by $\e$ and letting $\e\to 0$ that 
\begin{equation*}
\begin{aligned}
&\frac{1}{\tau}\ \left[ \int_\R(\id-\p_x\varphi)\ (\xi\circ\p_x\varphi)\ f_0\, dx + \frac{R}{R_\mu}\ \int_\R (\id-\p_x\psi)\ (\eta\circ\p_x\psi)\ g_0\, dx \right]\\
&+(1+R)\ \int_\R \p_x\xi\ \frac{f^2}{2}\, dx + R\ \int_\R \p_x\eta\ \frac{g^2}{2}\, dx + R\ \int_\R \left[ f\ \p_x(\eta g) + g\ \p_x(\xi f) \right]\, dx\leq 0.
\end{aligned}
\end{equation*}
Since the relation is valid for $(\xi,\eta)$ as well as for $(-\xi,-\eta)$, we end up with  
\begin{equation}\label{gau}
\begin{aligned}
&\frac{1}{\tau}\ \left[ \int_\R(\id-\p_x\varphi)\ (\xi\circ\p_x\varphi)\ f_0\, dx + \frac{R}{R_\mu}\ \int_\R (\id-\p_x\psi)\ (\eta\circ\p_x\psi)\ g_0\, dx \right]\\
&+(1+R)\ \int_\R \p_x\xi\ \frac{f^2}{2}\, dx + R\ \int_\R \p_x\eta\ \frac{g^2}{2}\, dx + R\ \int_\R \left[ f\ \p_x(\eta g) + g\ \p_x(\xi f) \right]\, dx = 0
\end{aligned}
\end{equation}
 for all $(\xi,\eta)\in C_0^\infty(\R;\R^2)$. 

Consider now $\Xi\in C_0^\infty(\R)$. For $x\in\R$, we have 
\begin{align*}
\left| \Xi(x)-\Xi(\p_x\varphi(x)) - \p_x\Xi(\p_x\varphi(x))\ (x-\p_x\varphi(x)) \right| =& \left| \int_{\p_x\varphi(x)}^x (x-y)\ \p_x^2\Xi(y)\ dy \right| \\
\le& \|\p_x^2\Xi\|_\infty\ \frac{(x-\p_x\varphi(x))^2}{2}\,.
\end{align*}
Multiplying the above inequality by $f_0(x),$ integrating over $\R,$ and using the definition of $\p_x\varphi$ yield
\begin{equation}\label{ga1}
\left| \int_\R \left[ \Xi(x)-\Xi(\p_x\varphi(x)) - \p_x\Xi(\p_x\varphi(x))\ (x-\p_x\varphi(x)) \right]\ f_0(x)\, dx \right| \le \|\p_x^2\Xi\|_\infty\ \frac{W_2^2(f,f_0)}{2}\,.
\end{equation}
Owing to \eqref{gau} with $(\xi,\eta)=(\p_x \Xi,0)$ and the property $f=\p_x\varphi \# f_0$, we deduce that 
$$
\left| \frac{1}{\tau}\ \int_\R \left(f - f_0 \right)\ \Xi\, dx - (1+R)\ \int_\R \frac{f^2}{2}\ \p_x^2\Xi\, dx - R\ \int_\R g\ \p_x\left( f\ \p_x\Xi \right)\, dx\right|\leq \frac{1}{2}\|\p_x^2\Xi\|_\infty\frac{W_2^2(f,f_0)}{\tau}.
$$
Taking into account that $(f,g)\in H^1(\R;\R^2)$ by Lemma~\ref{Unimin}, we arrive,  after integrating by parts once, to \eqref{eq:sup6}. A similar argument leads to \eqref{eq:sup7}.
\end{proof}

 We next develop further an argument from the proof of \cite[Proposition~2]{Ot98} which allows us to gain regularity on $f$ and $g$ by using the Euler-Lagrange equation.

%%%%%%%%%%%%%%%%%%%%%%%%%%%%%%%%%%%%%%%%%%%%%%%%%%%%%%%%%%%%%%%%%%%%%%%%%%%
\begin{cor}\label{co:sup1}
The functions $\sqrt{f}\ \p_x((1+R)f+Rg)$ and $\sqrt{g}\ \p_x(f+g)$ both belong to $L^2(\R)$ and 
\begin{subequations}\label{eq:sup10}
\begin{eqnarray}
\tau\ \left\| \sqrt{f}\ \p_x[(1+R)f+Rg] \right\|_2 & \le & W_2(f,f_0)\,, \label{eq:sup10a} \\
\tau\ R_\mu \left\| \sqrt{g}\ \p_x(f+g) \right\|_2 & \le & W_2(g,g_0)\,. \label{eq:sup10b}
\end{eqnarray}
\end{subequations}
\end{cor}
%%%%%%%%%%%%%%%%%%%%%%%%%%%%%%%%%%%%%%%%%%%%%%%%%%%%%%%%%%%%%%%%%%%%%%%%%%%

\medskip

It is worth mentioning here that the estimates \eqref{eq:sup10} match exactly the regularity of $(f,g)$ given by the dissipation in the energy inequality \eqref{dissE2}.

\medskip

\begin{proof}
Consider $\xi\in C_0^\infty(\R).$ We infer from \eqref{gau} with $\eta=0$ that, after integrating by parts,
$$
\int_\R \left[ (1+R)\ f\ \p_x f + R\ f\ \p_x g \right]\ \xi\, dx = \frac{1}{\tau}\ \int_\R (x-\p_x\varphi(x))\ (\xi\circ\p_x\varphi)(x)\ f_0(x)\, dx\,.
$$
Since $f=\p_x\varphi\#f_0$, it follows from the Cauchy-Schwarz inequality and \eqref{Brenier} that 
\begin{eqnarray*}
& & \left| \int_\R (x-\p_x\varphi(x))\ (\xi\circ\p_x\varphi)(x)\ f_0(x)\, dx \right| \\
& & \hspace{1cm} \le \left( \int_\R (x-\p_x\varphi(x))^2\ f_0(x)\, dx \right)^{1/2}\ \left( \int_\R (\xi\circ\p_x\varphi)^2(x)\ f_0(x)\, dx \right)^{1/2} \\
& &\hspace{1cm} \le W_2(f,f_0)\ \left( \int_\R \xi^2(x)\ f(x)\, dx \right)^{1/2}\,.
\end{eqnarray*}
Therefore,
\begin{equation}
\left| \int_\R \left[ (1+R)\ f\ \p_x f + R\ f\ \p_x g \right]\ \xi\, dx \right| \le \frac{W_2(f,f_0)}{\tau}\ 
\left( \int_\R \xi^2(x)\ f(x)\, dx \right)^{1/2}\,.\label{eq:sup11}
\end{equation}
Consider next a nonnegative function $\chi\in C_0^\infty(\R)$ such that $\|\chi\|_1=1$ and define $\chi_m(x):=m\chi(mx)$ for $m\ge 1$ and $x\in\R$. Then, $(\chi_m)_{m\ge 1}$ is a sequence of mollifiers in $\R$ and, given $\vartheta\in C_0^\infty(\R)$ and $m\ge 1$, the function $\vartheta/(m^{-1/4} + \chi_m*f)^{1/2}$ belongs to $C_0^\infty(\R)$. Taking $\xi=\vartheta/(m^{-1/4} + \chi_m*f)^{1/2}$ in \eqref{eq:sup11}, we obtain
$$
\left| \int_\R \frac{f\ \p_x[(1+R)f+Rg]}{\sqrt{m^{-1/4} + \chi_m*f}}\ \vartheta\, dx \right| \le \frac{W_2(f,f_0)}{\tau}\ \left\| \frac{f}{m^{-1/4}+\chi_m*f} \right\|_\infty^{1/2}\ \|\vartheta\|_2\,.
$$
The previous inequality being valid for all $\vartheta\in C_0^\infty(\R)$, a duality argument yields
\begin{equation}
\left\| \frac{f\ \p_x[(1+R)f+Rg]}{\sqrt{m^{-1/4} + \chi_m*f}} \right\|_2 \le \frac{W_2(f,f_0)}{\tau}\ \left\| \frac{f}{m^{-1/4}+\chi_m*f} \right\|_\infty^{1/2}\,. \label{eq:sup12}
\end{equation}
Now, since $f\in H^1(\R)$ by Lemma~\ref{Unimin}, we have $\|\chi_m*f-f\|_\infty\le C_\chi\ \|\p_x f\|_2\ m^{-1/2}$ for some constant $C_\chi>0$ depending only on $\chi$ from which we deduce that
\begin{equation}
\left\| \frac{f}{m^{-1/4} + \chi_m*f} \right\|_\infty \le \left\| \frac{f-\chi_m*f}{m^{-1/4} + \chi_m*f} \right\|_\infty + \left\| \frac{\chi_m*f}{m^{-1/4} + \chi_m*f} \right\|_\infty \le 1+ \frac{C_\chi\ \|\p_x f\|_2}{m^{1/4}}\,.\label{eq:sup13}
\end{equation}
In particular, for $x\in\R$,
$$
\left| \frac{f(x)}{\sqrt{m^{-1/4} + \chi_m*f(x)}} \right| \le \left( 1 + \sqrt{C_\chi\ \|\p_x f\|_2} \right)\ \sqrt{f(x)} \in L^2(\R)
$$
and
$$
\lim_{m\to\infty} \frac{f(x)}{\sqrt{m^{-1/4} + \chi_m*f(x)}} = \left\{
\begin{array}{ccl}
0 = \sqrt{f(x)} & \text{ if } & f(x)=0\,,\\
 & &\\
\sqrt{f(x)} & \text{ if } & f(x)>0\,,
\end{array}
\right.
$$
so that
$$
\frac{f}{\sqrt{m^{-1/4} + \chi_m*f}} \longrightarrow \sqrt{f} \quad\text{ in }\quad L^2(\R)
$$
by the Lebesgue dominated convergence theorem. Since $(1+R)f+Rg$ belongs to $H^1(\R)$ by Lemma~\ref{Unimin}, we conclude that 
\begin{equation}
\frac{f}{\sqrt{m^{-1/4} + \chi_m*f}}\ \p_x[(1+R)f+Rg] \longrightarrow \sqrt{f}\ \p_x[(1+R)f+Rg] \quad\text{ in }\quad L^1(\R)\,. \label{eq:sup14}
\end{equation}
Owing to \eqref{eq:sup13} and \eqref{eq:sup14}, we may let $m\to\infty$ in \eqref{eq:sup12} and deduce that $\sqrt{f}\ \p_x[(1+R)f+Rg]\in L^2(\R)$ and satisfies \eqref{eq:sup10a}. The proof of \eqref{eq:sup10b} is similar.
\end{proof}

%%%%%%%%%%%%%%%%%%%%%%%%%%%%%%%%%%%%%%%%%%%%%%%%%%%%%%%%%%%%%%%%%%%%%%%%%%%
\subsection{Interpolation}\label{s:inter}
%%%%%%%%%%%%%%%%%%%%%%%%%%%%%%%%%%%%%%%%%%%%%%%%%%%%%%%%%%%%%%%%%%%%%%%%%%%

Thanks to the results established in the previous sections, we are now in a position to set up a variational scheme to approximate the solution
 to \eqref{eq:problem}. More precisely, given $(f_0,g_0)\in\cK_2$ and $\tau\in (0,1)$, we define inductively a sequence $(f_\tau^{n}, g_\tau^{n})_{n\ge 0}$ as follows:
\begin{eqnarray}
(f_\tau^0,g_\tau^0) & := & (f_0,g_0)\,, \label{eq:sup20} \\
\cF_\tau^n\left( f_\tau^{n+1},g_\tau^{n+1} \right) & := & \inf_{(u,v)\in\cK_2}{\cF_\tau^n(u,v)}\,, \label{eq:sup21}
\end{eqnarray}
with
$$
\cF_\tau^n(u,v) := \frac{1}{2\tau}\ \left(W_2^2\left( u,f_\tau^n \right) + \frac{R}{R_\mu}\ W_2^2\left( v,g_\tau^n\right) \right) + \E(u,v)\,, \quad (u,v)\in\cK_2\,,
$$
the existence and uniqueness of $\left( f_\tau^{n+1},g_\tau^{n+1} \right)$ being guaranteed by Lemma~\ref{Unimin} for each $n\ge 0$. We next define two interpolation functions $f_\tau$ and $g_\tau$ by \eqref{eq:interp}, i.e. $f_\tau(t):=f_\tau^n$ and $g_\tau(t):=g_\tau^n$ for $t\in [n\tau,(n+1)\tau)$ and $n\ge 0.$ By Lemma~\ref{le:sup1}, we have
\begin{equation}\label{System-n}
\left\{
\begin{aligned}
&\left| \int_\R \left( f^{n}_\tau - f^{n-1}_\tau \right)\ \xi\, dx + \tau\ \int_\R f^{n}_\tau\ \p_x\left( (1+R)\ f^{n}_\tau + R\ g_\tau^n \right)\ \p_x\xi\, dx \right| \leq \frac{\|\p_x^2\xi\|_\infty}{2}\ W_2^2(f_\tau^n,f_\tau^{n-1}),\\[1ex]
&\left| \int_\R \left( g^{n}_\tau - g^{n-1}_\tau \right)\ \xi\, dx + \tau\ R_\mu\ \int_\R g^{n}_\tau\ \p_x\left( g^{n}_\tau + g_\tau^n \right)\ \p_x\xi\, dx \right| \leq \frac{\|\p_x^2\xi\|_\infty}{2}\ W_2^2(g^{n}_\tau,g^{n-1}_\tau),
\end{aligned}\right.
\end{equation}
for all $n\geq1$ and $\xi\in C^\infty_0(\R).$
 Given $T>0$ arbitrary,  we set $N:=[T/\tau].$
 Summing both equations of \eqref{System-n} from $n=1$ to $n=N,$  we find
\begin{equation}\label{Sumf}
\left| \int_\R \left( f_\tau(T) - f_0 \right)\ \xi\, dx \right.  + \left. \int_{\tau}^{(N+1)\tau} \int_\R f_\tau\ \p_x\left( (1+R)\ f_\tau + R\ g_\tau \right)\ \p_x\xi\, dxdt \right| \leq \frac{\|\p_x^2\xi\|_\infty}{2}\ \sum_{n=1}^{N} W_2^2(f^{n}_\tau,f^{n-1}_\tau),
\end{equation}
\begin{equation}\label{Sumg}
\begin{aligned}
\left|\int_\R \left( g_\tau(T) - g_0 \right)\ \xi\, dx \right.  + \left. R_\mu\ \int_{\tau}^{(N+1)\tau} \int_\R g_\tau\ \p_x\left( f_\tau + g_\tau \right)\ \p_x\xi\, dxdt \right| \leq \frac{\|\p_x^2\xi\|_\infty}{2}\ \sum_{n=1}^{N} W_2^2(g^{n}_\tau,g^{n-1}_\tau).
\end{aligned}
\end{equation}

%%%%%%%%%%%%%%%%%%%%%%%%%%%%%%%%%%%%%%%%%%%%%%%%%%%%%%%%%%%%%%%%%%%%%%%%%%%
%%%%%%%%%%%%%%%%%%%%%%%%%%%%%%%%%%%%%%%%%%%%%%%%%%%%%%%%%%%%%%%%%%%%%%%%%%%
\section{Convergence}\label{s:conv}
%%%%%%%%%%%%%%%%%%%%%%%%%%%%%%%%%%%%%%%%%%%%%%%%%%%%%%%%%%%%%%%%%%%%%%%%%%%
%%%%%%%%%%%%%%%%%%%%%%%%%%%%%%%%%%%%%%%%%%%%%%%%%%%%%%%%%%%%%%%%%%%%%%%%%%%

We gather in the next lemma various properties of the interpolations $(f_\tau,g_\tau)$ defined 
in Section~\ref{s:inter} which are consequences of Lemma~\ref{Unimin} and Corollary~\ref{co:sup1}.

%%%%%%%%%%%%%%%%%%%%%%%%%%%%%%%%%%%%%%%%%%%%%%%%%%%%%%%%%%%%%%%%%%%%%%%%%%%
\begin{lemma}\label{L:Est} There exists a positive constant $C_1$ depending only on 
$R$, $R_\mu$, $f_0$, and $g_0$ such that, for all $t\ge 0$ and $\tau\in (0,1)$, we have
\begin{align}
\label{e1}
 (i)\qquad & \int_\R f_\tau(t)\, dx=\int_\R g_\tau(t)\, dx=1,\\
\label{e2}
 (ii)\qquad & \sum_{n=1}^\infty \left[ W_2^2(f_\tau^n,f_\tau^{n-1}) + W_2^2(g_\tau^n,g_\tau^{n-1}) \right] \leq C_1 \tau , \\
 \label{e3}
 (iii)\qquad & \E(f_\tau(t),g_\tau(t))\leq \E(f_\tau(s),g_\tau(s)), \quad s\in [0,t],\\
 \label{e4}
 (iv)\qquad & \int_\R \left( f_\tau + g_\tau \right)(t,x)\ x^2\, dx\leq C_1\ (1+t),\\
 \label{e5}
 (v)\qquad & \int_\tau^t \left[ \|\p_x f_\tau(s)\|_2^2 + \|\p_x g_\tau(s)\|_2^2 \right]\, ds \leq C_1\ (1+t),\\
 \label{e6}
 (vi)\qquad& \int_\tau^\infty \int_\R f_\tau\ \left| \p_x\left[ (1+R)\ f_\tau + R\ g_\tau \right] \right|^2\, dxds \leq C_1,\\
\label{e7}
  (vii)\qquad& \int_\tau^\infty \int_\R g_\tau\ \left| \p_x\left( f_\tau + g_\tau \right) \right|^2\, dxds \leq C_1.
\end{align}
\end{lemma}
%%%%%%%%%%%%%%%%%%%%%%%%%%%%%%%%%%%%%%%%%%%%%%%%%%%%%%%%%%%%%%%%%%%%%%%%%%%

\begin{proof} The property \eqref{e1} readily follows from the fact that $(f_\tau^n,g_\tau^n)\in\cK_2$ for all $n\ge 0$ and $\tau>0$. 
Next, for $\tau>0$ and $n\ge 1$, the minimizing property of $(f_\tau^n,g_\tau^n)$ ensures that
\begin{equation}\label{eq:qe}
\E(f_\tau^n,g_\tau^n)+\frac{1}{2\tau}\left[ W_2^2(f_\tau^n,f_\tau^{n-1}) + \frac{R}{R_\mu}\ W_2^2(g_\tau^n,g_\tau^{n-1}) \right] \leq \E(f_\tau^{n-1},g_\tau^{n-1}).
\end{equation}
 Given $t\in (0,\infty)$ and $s\in [0,t]$, we set $N:=[t/\tau]$, $\nu:=[s/\tau]$, and sum \eqref{eq:qe} from $n=\nu+1$ up to $n=N$ to obtain, 
since $(f_\tau,g_\tau)(t)=(f_\tau^N,g_\tau^N)$ and $(f_\tau,g_\tau)(s)=(f_\tau^\nu,g_\tau^\nu)$,
\begin{equation}\label{eq:qee}
\E(f_\tau(t),g_\tau(t)) + \frac{1}{2\tau}\ \sum_{n=\nu+1}^N \left[ W_2^2(f_\tau^n,f_\tau^{n-1})+\frac{R}{R_\mu}\ W_2^2(g_\tau^n,g_\tau^{n-1}) \right] 
\leq \E(f_\tau(s),g_\tau(s)).
\end{equation}
The monotonicity property \eqref{e3} is a straightforward consequence of \eqref{eq:qee} while the nonnegativity of $\E$ and \eqref{eq:qee} with $s=\nu=0$ give
$$
\sum_{n=1}^N \left[ W_2^2(f_\tau^n,f_\tau^{n-1}) + \frac{R}{R_\mu}\ W_2^2(g_\tau^n,g_\tau^{n-1}) \right] \leq 2 \E(f_0,g_0)\ \tau\,.
$$
Since the right-hand side of the above inequality does not depend on $N$, we obtain \eqref{e2}. In order to prove \eqref{e4}, we
 combine \eqref{eq:1} and \eqref{e2} and obtain for $t\ge 0$ with $N:=[t/\tau]$ 
\begin{eqnarray*}
\int_R f_\tau(t,x)\ x^2\, dx & = & \int_\R f_\tau^N(x)\ x^2\, dx \leq 2\int_\R f_0(x)\ x^2\, dx + 2 W_2^2(f_\tau^N,f_0) \\ 
& \leq & 2 \int_\R f_0(x)\ x^2\, dx + 2 N\ \sum_{n=1}^N W_2^2(f_\tau^n,f_\tau^{n-1}) \\
& \le & 2 \int_\R f_0(x)\ x^2\, dx +  4N \tau\ \E(f_0,g_0) \le C\ (1+t)\,.
\end{eqnarray*}

We next infer from \eqref{Heat3} that, for $n\ge 1$,
\[
\tau\ \left( \|\p_x f_\tau^{n}\|_2^2 + R\ \|\p_x (f_\tau^{n}+g_\tau^{n})\|_2^2 \right) \leq H(f_\tau^{n-1}) - H(f_\tau^{n}) + \frac{R}{R_\mu}\ 
\left( H(g_\tau^{n-1}) - H(g_\tau^{n}) \right).
\]
 Let $N\ge 1$. Summation from $n=1$ to $N$ yields
\begin{align}
&\int_\tau^{(N+1)\tau} \left( \|\p_xf_\tau(s)\|_2^2 + R\ \|\p_x(f_\tau+g_\tau)(s)\|_2^2 \right)\, ds  \leq  H(f_0)-H(f_\tau(N\tau))  + \frac{R}{R_\mu}\ \left(H(g_0)-H(g_\tau(N\tau)) \right).\label{Hes}
\end{align}
It now follows from Lemma~\ref{le:ap1}, \eqref{e1}, \eqref{e4}, and \eqref{Hes} that
\begin{align*}
&\int_\tau^{(N+1)\tau} \left( \|\p_xf_\tau(s)\|_2^2 + R\ \|\p_x(f_\tau+g_\tau)(s)\|_2^2 \right)\, ds \nonumber\\
& \hspace{1cm} \leq  H(f_0) + \frac{R}{R_\mu}\ H(g_0) + \frac{(R+R_\mu) C_\ell}{R_\mu} + \int_\R (1+x^2)\ \left( f_\tau(N\tau) 
+ \frac{R}{R_\mu}\ g_\tau(N\tau) \right)\leq C\ (1+ N\tau)\,,
\end{align*}
which entails the validity of \eqref{e5} for $t\in [N\tau,(N+1)\tau)$.

We finish the  proof by showing \eqref{e6} and \eqref{e7}. 
 By Corollary~\ref{co:sup1}, we have for $n\ge 1$
$$
\tau^2\ \left\| \sqrt{f_\tau^n}\ \p_x\left[ (1+R)\ f_\tau^n + R\ g_\tau^n \right] \right\|_2^2 \le W_2^2(f_\tau^n,f_\tau^{n-1})\,.
$$
Summing over $n\ge 1$ and using \eqref{e2} give
$$
\sum_{n=1}^\infty \tau^2\ \left\| \sqrt{f_\tau^n}\ \p_x\left[ (1+R)\ f_\tau^n + R\ g_\tau^n \right] \right\|_2^2 
\le \sum_{n=1}^\infty W_2^2(f_\tau^n,f_\tau^{n-1}) \le C_1\tau\,,
$$
whence \eqref{e6}. The proof of \eqref{e7} also relies on Corollary~\ref{co:sup1} and is similar.
\end{proof}

%%%%%%%%%%%%%%%%%%%%%%%%%%%%%%%%%%%%%%%%%%%%%%%%%%%%%%%%%%%%%%%%%%%%%%%%%%%
%%%%%%%%%%%%%%%%%%%%%%%%%%%%%%%%%%%%%%%%%%%%%%%%%%%%%%%%%%%%%%%%%%%%%%%%%%%
\subsection{Compactness}
%%%%%%%%%%%%%%%%%%%%%%%%%%%%%%%%%%%%%%%%%%%%%%%%%%%%%%%%%%%%%%%%%%%%%%%%%%%
%%%%%%%%%%%%%%%%%%%%%%%%%%%%%%%%%%%%%%%%%%%%%%%%%%%%%%%%%%%%%%%%%%%%%%%%%%%

 We now turn to the compactness properties of $(f_{\tau})_{\tau>0}$ and $(g_{\tau})_{\tau>0}$ 
and point out that the nonlinearity of \eqref{eq:S2} requires strong compactness. 
 We first observe that the compactness with respect to the space variable $x$ 
is granted by \eqref{e5}  thanks to the following lemma.

%%%%%%%%%%%%%%%%%%%%%%%%%%%%%%%%%%%%%%%%%%%%%%%%%%%%%%%%%%%%%%%%%%%%%%%%%%%
\begin{lemma}\label{le:comp}
The spaces $H^1(\R)\cap L^1(\R,(1+x^2)\,dx)$ and $L^2(\R)\cap L^1(\R,(1+x^2)\,dx)$ are compactly embedded in $L^2(\R)$ and $H^{-3}(\R)$, respectively.
\end{lemma}
%%%%%%%%%%%%%%%%%%%%%%%%%%%%%%%%%%%%%%%%%%%%%%%%%%%%%%%%%%%%%%%%%%%%%%%%%%%

\begin{proof}
Let us first consider a bounded sequence $(h_i)_{i\ge 1}$ in $H^1(\R)\cap L^1(\R,(1+x^2)\,dx)$. 
On the one hand, since $H^1(\R)$ is continuously embedded in $L^\infty(\R)$ and $C^{1/2}(\R)$, 
the Arzel{\`a}-Ascoli theorem implies that there are $h\in H^1(\R)$ and a subsequence  of $(h_i)_{i\ge 1}$ (not relabeled),
 such that $(h_i)_{i\ge 1}$ converges to $h$ in $C([-R,R])$ for all $R>0$. 
On the other hand, using once more the embedding of $H^1(\R)$ in $L^\infty(\R)$, we have  for $R>1$
\begin{align*}
\int_{\R}|h_i(x)-h(x)|^2\, dx\leq& \int_{\{|x|\leq R\}} |h_i(x)-h(x)|^2\, dx + \int_{\{|x|> R\}} |h_i(x)-h(x)|^2\, dx\\
\leq&2R\ \|h_i-h\|^2_{C([-R,R])} + \frac{1}{R^2}\ \|h_i-h\|_\infty\ \int_\R x^2\ |h_i(x)-h(x)|\, dx\\
\leq &2R\ \|h_i-h\|^2_{C([-R,R])} + \frac{2}{R^2}\ \sup_{i\ge 1}\left\{ \|h_i\|_\infty\ \int_\R x^2\ |h_i(x)|\, dx \right\}\,.
\end{align*}
 Letting first $i\to\infty$ and then $R\to \infty$ shows that $(h_i)_{i\ge 1}$ converges to $h$ in $L^2(\R)$.

Next, let $(h_i)_{i\ge 1}$ be a bounded sequence in $L^2(\R)\cap L^1(\R,(1+x^2)\,dx)$ and denote the Fourier transform of $h_i$ by 
$\mathcal{F}h_i$ for $i\ge 1$. A straightforward consequence of the bounds for $(h_i)_{i\ge 1}$ is that $(\mathcal{F}h_i)_{i\ge 1}$ 
is bounded in $L^2(\R)\cap W^{2,\infty}(\R)$. Arguing as above, this implies that $(\mathcal{F}h_i)_{i\ge 1}$ is relatively compact 
in $L^2(\R, (1+x^2)^{-3}\, dx)$. Coming back to the original variable, $(h_i)_{i\ge 1}$ is relatively compact in $H^{-3}(\R)$ as claimed.
\end{proof}

We next  turn to the compactness in time and prove the following result:

%%%%%%%%%%%%%%%%%%%%%%%%%%%%%%%%%%%%%%%%%%%%%%%%%%%%%%%%%%%%%%%%%%%%%%%%%%%
\begin{lemma}\label{le:sup3}
There is a positive constant $C_2$ depending only on $R$, $R_\mu$, $f_0$, and $g_0$ such that, for $\tau\in (0,1)$ and $(t,s)\in [0,\infty)\times [0,\infty)$,
\begin{equation}
\| f_\tau(t) - f_\tau(s) \|_{H^{-3}} + \| g_\tau(t) - g_\tau(s) \|_{H^{-3}} \le C_2\  \sqrt{|t-s| + \tau}. \label{eq:sup100}
\end{equation}
\end{lemma}
%%%%%%%%%%%%%%%%%%%%%%%%%%%%%%%%%%%%%%%%%%%%%%%%%%%%%%%%%%%%%%%%%%%%%%%%%%%

\begin{proof}
Consider $t\in (0,\infty)$, $s\in [0,t]$, and  define the integers $N:=[t/\tau]$ and $\nu:=[s/\tau]$. 
Either $N=\nu$ and $f_\tau(t)-f_\tau(s)=0$ satisfies \eqref{eq:sup100} or $N\ge \nu+1$ and it follows 
from \eqref{System-n} that, for $n\in\{\nu+1,\cdots,N\}$ and $\xi\in C_0^\infty(\R)$,
\begin{align*}
\left| \int_\R (f^n_{\tau}-f^{n-1}_{\tau})\ \xi\, dx \right| \le &\int_{n\tau}^{(n+1)\tau} \int_\R  f_\tau(s)\ 
\left| \p_x\left[ (1+R)\ f_\tau + R\ g_\tau \right](s) \right|\ \left| \p_x\xi \right|\, dx\,ds \\
& + \frac{\|\p_x^2\xi\|_\infty}{2}\ W_2^2(f^n_\tau,f^{n-1}_\tau)
\end{align*}
Summing the above inequality from $n=\nu+1$ to $n=N$ and using \eqref{e1}, \eqref{e3}, \eqref{e6}, and the Cauchy-Schwartz inequality, we are led to 
\begin{eqnarray*}
\left| \int_\R (f_\tau(t)-f_\tau(s))\ \xi\, dx \right| & = & \left| \int_\R (f^N_{\tau}-f^{\nu}_{\tau})\ \xi\, dx \right| 
\le \sum_{n=\nu+1}^N \left| \int_\R (f^n_{\tau}-f^{n-1}_{\tau})\ \xi\, dx \right| \\
& \le & \int_{(\nu+1)\tau}^{(N+1)\tau} \int_\R f_\tau(s)\ \left| \p_x\left[ (1+R)\ f_\tau + R\ g_\tau \right](s) \right|\ \left| \p_x\xi \right|\, dx \, ds\\ 
& & + \frac{\|\p_x^2\xi\|_\infty}{2}\ \sum_{n=\nu+1}^N W_2^2(f^n_\tau,f^{n-1}_\tau) \\
& \le & \|\p_x\xi\|_\infty\ \int_{(\nu+1)\tau}^{(N+1)\tau} \|f_\tau(s)\|_1^{1/2}\ \left\| \sqrt{f_\tau}\ \p_x\left[ (1+R)\ f_\tau + R\ g_\tau \right](s) \right\|_2\, ds \\
& & + C_1\ \tau\ \|\p_x^2\xi\|_\infty \\
& \le & C\ \|\xi\|_{W^{2,\infty}}\ \left( \sqrt{(N-\nu)\tau} + \tau \right) \\ 
& \le & C\ \|\xi\|_{W^{2,\infty}}\ \left( \sqrt{t-s+\tau} + \tau \right).
\end{eqnarray*}
Since $H^3(\R)$ is continuously embedded in $W^{2,\infty}(\R)$, the claimed
 estimate for $f_\tau(t)-f_\tau(s)$ follows by a density argument.
 A similar computation relying on \eqref{System-n}, \eqref{e1}, \eqref{e3}, and \eqref{e7} gives the same 
estimate for $g_\tau(t)-g_\tau(s)$ and completes the proof of Lemma~\ref{le:sup3}.
\end{proof}

 We are now in a position to establish the strong compactness of $(f_\tau,g_\tau)_{\tau>0}$ in $L^2((0,T)\times\R)$ for all $T>0$ as announced in \eqref{T1}. 

%%%%%%%%%%%%%%%%%%%%%%%%%%%%%%%%%%%%%%%%%%%%%%%%%%%%%%%%%%%%%%%%%%%%%%%%%%%
\begin{lemma}\label{le:scomp} 
There are a sequence $(\tau_k)_{k\ge 1}$, $\tau_k\to 0$, and functions $f$ and $g$ in $C([0,\infty);H^{-3}(\R))$ such that, for all $t\ge 0$,
\begin{align}
& & \left( f_{\tau_k}(t),g_{\tau_k}(t) \right) \longrightarrow (f(t),g(t)) \quad \text{ in } \quad  H^{-3}(\R;\R^2)\,, \label{be0} \\
& & \left( f_{\tau_k},g_{\tau_k} \right) \longrightarrow (f,g) \quad \text{ in } \quad L^2((0,t)\times\R;\R^2)\,, \label{be1} \\
& & \left( f_{\tau_k},g_{\tau_k} \right) \longrightarrow (f,g) \quad \text{ a.e. in } \quad (0,\infty)\times\R\,. \label{be1b}
\end{align}
\end{lemma}
%%%%%%%%%%%%%%%%%%%%%%%%%%%%%%%%%%%%%%%%%%%%%%%%%%%%%%%%%%%%%%%%%%%%%%%%%%%

\begin{proof}
The proof relies on \cite[Proposition~3.3.1]{AGS08} and \cite[Lemma~9]{Si87}. 
Indeed, it first follows from \eqref{e1}, \eqref{e3}, \eqref{e4}, and Lemma~\ref{le:comp} that $(f_\tau(t))_{\tau\in (0,1)}$
 lies in a compact subset of $H^{-3}(\R)$. This fact, together with Lemma~\ref{le:sup3} and a refined version of the 
Arzel\`a-Ascoli theorem \cite[Proposition~3.3.1]{AGS08} ensures that  there are a sequence $(\tau_k)_{k\ge 1}$, $\tau_k\to 0$, 
and a function $f\in C([0,\infty);H^{-3}(\R))$ such that $\left( f_{\tau_k}(t) \right)$ converges towards $f(t)$ in $H^{-3}(\R;\R^2)$  for each $t\ge 0$. 
Since the same argument applies for $(g_\tau)_{\tau\in (0,1)}$, we have established \eqref{be0}. 
We then infer from \eqref{e3}, the embedding of $L^2(\R)$ in $H^{-3}(\R)$, the convergence \eqref{be0}, and the Lebesgue dominated convergence theorem that
\begin{equation}
\left( f_{\tau_k},g_{\tau_k} \right) \longrightarrow (f,g) \quad \text{ in } \quad  L^2(0,T;H^{-3}(\R;\R^2))
\;\;\text{ for all }\;\; T>0\,. \label{pim}
\end{equation} 
Now, given $\delta\in (0,1)$ and $T>1$, the estimates \eqref{e1}, \eqref{e3} (with $s=0$), \eqref{e4}, and \eqref{e5} in Lemma~\ref{L:Est} ensure that 
\begin{equation}
\left( f_{\tau_k},g_{\tau_k} \right)_{k\geq 1} \;\;\text{ is bounded in }\;\; L^2(\delta,T;H^1(\R)\cap L^1(\R,(1+x^2)\,dx))\,.\label{pam}
\end{equation}
Since $H^1(\R)\cap L^1(\R,(1+x^2)\,dx)$ is compactly embedded in $L^2(\R)$ by Lemma~\ref{le:comp} and $L^2(\R)$ is continuously embedded in $H^{-3}(\R)$, we are in a position to apply \cite[Lemma~9]{Si87} and deduce from \eqref{pim} and 
\eqref{pam} that $\left( f_{\tau_k},g_{\tau_k} \right)_{k\ge 1}$ converges towards $(f,g)$ in $L^2((\delta,T)\times \R;\R^2)$.
Owing to \eqref{e3}, this convergence may actually be improved to \eqref{be1}. 
The a.e. convergence \eqref{be1b} then follows from \eqref{be1} after possibly extracting a further subsequence.
\end{proof}

Finally, \eqref{e5} implies that, after possibly extracting a further subsequence, we may assume that
\begin{equation}
\left( \p_x f_{\tau_k}, \p_x g_{\tau_k} \right) \rightharpoonup (\p_x f, \p_x g) \quad \text{ in } \quad L^2((\delta,T)\times\R) \quad\text{for all }
\quad 0<\delta<T\,. \label{be2}
\end{equation}
Now, combining \eqref{e6}, \eqref{e7}, \eqref{be1}, and \eqref{be2},  we obtain
\begin{equation}
\left\{
\begin{array}{ccl}
\sqrt{f_{\tau_k}}\ \p_x\left[ (1+R)\ f_{\tau_k} + R\ g_{\tau_k} \right] & \rightharpoonup & \sqrt{f}\ \p_x\left[ (1+R)\ f + R\ g \right] \\
\sqrt{g_{\tau_k}}\ \p_x\left( f_{\tau_k} + g_{\tau_k} \right) & \rightharpoonup & \sqrt{g}\ \p_x\left( f + g \right) 
\end{array}
\right.
\quad \text{ in } \quad L^2((\delta,T)\times\R) \label{be3} 
\end{equation}
for $0<\delta<T,$ while \eqref{be1} and \eqref{be2} imply that,  for $0<\delta<T,$
\begin{equation}
\left\{
\begin{array}{ccl}
f_{\tau_k}\ \p_x\left[ (1+R)\ f_{\tau_k} + R\ g_{\tau_k} \right] & \rightharpoonup & f\ \p_x\left[ (1+R)\ f + R\ g \right] \\
g_{\tau_k}\ \p_x\left( f_{\tau_k} + g_{\tau_k} \right) & \rightharpoonup & g\ \p_x\left( f + g \right) 
\end{array}
\right.
\quad \text{ in } \quad L^1((\delta,T)\times\R)\,. \label{be4}
\end{equation}

%%%%%%%%%%%%%%%%%%%%%%%%%%%%%%%%%%%%%%%%%%%%%%%%%%%%%%%%%%%%%%%%%%%%%%%%%%%
\subsection{Passing to the limit}\label{se:pttl}
%%%%%%%%%%%%%%%%%%%%%%%%%%%%%%%%%%%%%%%%%%%%%%%%%%%%%%%%%%%%%%%%%%%%%%%%%%%

Combining the convergence \eqref{be0}  with the estimates \eqref{e1}, \eqref{e3} (with $s=0$) and \eqref{e4} in Lemma~\ref{L:Est} ensures that $(f(t),g(t))\in\cK_2$ for all $t\ge 0$. Moreover, gathering \eqref{e3}, \eqref{e5}, \eqref{be1}, and \eqref{be2}, we conclude that $(f,g)$ satisfies the  integrability properties $(i)$ of  Theorem~\ref{MT:1}. In addition, it follows from \eqref{be0} and Lemma~\ref{le:sup3} that 
\begin{equation}
\| f(t)-f(s)\|_{H^{-3}}  + \| g(t)-g(s)\|_{H^{-3}} \le C_2\ \sqrt{|t-s|}\,, \qquad (t,s)\in [0,\infty)\times [0,\infty)\,, \label{be5}
\end{equation} 
which proves assertion $(ii)$  of  Theorem~\ref{MT:1}.

In order to establish the estimate $(b)$ of Theorem~\ref{MT:1}, we pick $T>0$ and set $N_k:=[T/\tau_k]$ for all integers $k\geq1.$
Then, we infer from Corollary~\ref{co:sup1} and \eqref{eq:qee} (with $s=0$) that for all $k\geq 1$ we have
\begin{align*}
 &\frac{1}{2}\ \int_{\tau_k}^T \left\{ \left\| \sqrt{f_{\tau_k}(\sigma)}\ \p_x\left[(1+R)f_{\tau_k} + R g_{\tau_k} \right] (\sigma) \right\|_2^2 + R R_\mu \left\| \sqrt{g_{\tau_k}(\sigma)}\ \p_x[f_{\tau_k} + g_{\tau_k}] (\sigma) \right\|_2^2 \right\}\, d\sigma\\
&\leq\sum_{n=1}^{N_k} \left[ \frac{W_2^2(f_{\tau_k}^n,f_{\tau_k}^{n-1})}{2\tau_k} + \frac{R}{R_\mu} \frac{W_2^2(g_{\tau_k}^n,g_{\tau_k}^{n-1})}{2\tau_k} \right]
\leq\E(f_0,g_0)-\E(f_{\tau_k}(T), g_{\tau_k}(T)).
\end{align*}
Letting $k\to\infty,$ the convergences \eqref{be1} and \eqref{be3} lead us to 
\begin{align*}
& \frac{1}{2}\ \int_{\delta}^T \left\{ \left\| \sqrt{f(\sigma)}\ \p_x\left[(1+R)f + R g \right](\sigma) \right\|_2^2 + R R_\mu \left\| \sqrt{g(\sigma)}\ \p_x[f + g](\sigma) \right\|_2^2 \right\}\, d\sigma \\
& \hspace{2cm} \leq \E(f_0,g_0)-\E(f(T), g(T)).
\end{align*}
for all $\delta\in (0,1)$, whence the desired assertion $(b)$ of Theorem~\ref{MT:1} after letting $\delta\to 0$.

\medskip

Now, we identify the equations solved by $f$ and $g$. To this end, fix $\xi\in C_0^\infty(\R)$, $t\in (0,\infty)$, $s\in (0,t)$ and set $N:=[t/\tau]$ and $\nu:=[s/\tau]$. We infer from \eqref{Sumf}, \eqref{e1}, \eqref{e2}, and \eqref{e6} that
\begin{align*}
& \left| \int_\R (f_\tau(t)-f_\tau(s))\ \xi\, dx + \int_s^t \int_\R f_\tau(\sigma)\ \p_x\left[ (1+R)\ f_\tau + R\ g_\tau \right](\sigma)\ \p_x \xi\, dxd\sigma \right| \\
 &\le  \left| \int_\R (f_\tau(t)-f_0)\ \xi\, dx + \int_\tau^{(N+1)\tau} \int_\R f_\tau(\sigma)\ \p_x\left[ (1+R)\ f_\tau + 
R\ g_\tau \right](\sigma)\ \p_x \xi\, dxd\sigma \right| \\
 &\phantom{=}+ \left| \int_\R (f_\tau(s)-f_0)\ \xi\, dx + \int_\tau^{(\nu+1)\tau} \int_\R f_\tau(\sigma)\ \p_x\left[ (1+R)\ f_\tau + 
R\ g_\tau \right](\sigma)\ \p_x \xi\, dxd\sigma \right| \\
&\phantom{=}+ \left| \int_t^{(N+1)\tau} \int_\R f_\tau(\sigma)\ \p_x\left[ (1+R)\ f_\tau + R\ g_\tau \right](\sigma)\ \p_x \xi\, dxd\sigma \right| \\
  &\phantom{=}+ \left| \int_s^{(\nu+1)\tau} \int_\R f_\tau(\sigma)\ \p_x\left[ (1+R)\ f_\tau + R\ g_\tau \right](\sigma)\ \p_x \xi\, dxd\sigma \right| \\
 &\le  \|\p_x^2\xi\|_\infty\ \sum_{n=1}^N W_2^2(f_\tau^n,f_\tau^{n-1}) \\
 &\phantom{=}+ \|\p_x\xi\|_\infty\ \int_s^{(\nu+1)\tau}  \left\| f_\tau(\sigma) \right\|_1^{1/2}\ \left\| \sqrt{f_\tau}\ \p_x\left[ (1+R)\ f_\tau + R\ g_\tau \right](\sigma) \right\|_2\, d\sigma \\
  &\phantom{=}+ \|\p_x\xi\|_\infty\ \int_t^{(N+1)\tau}  \left\| f_\tau(\sigma) \right\|_1^{1/2}\ \left\| \sqrt{f_\tau}\ \p_x\left[ (1+R)\ f_\tau + R\ g_\tau \right](\sigma) \right\|_2\, d\sigma \\
 &\le  C\ \|\xi\|_{W^{2,\infty}}\ \left( \tau + \sqrt{\tau} \right)\,.
\end{align*} 
Taking $\tau=\tau_k$ in the above inequality and letting $k\to\infty$ with the 
help of \eqref{be0} and \eqref{be4}, we end up with the first identity in \eqref{T2}
$$
\int_\R (f(t)-f(s))\ \xi\, dx + \int_s^t \int_\R f(\sigma)\ \p_x\left[ (1+R)\ f + R\ g \right](\sigma)\ \p_x \xi\, dxd\sigma = 0\,.
$$
The proof of the second one being similar, it remains to check the property $(a)$ stated in Theorem~\ref{MT:1}. To this end, we first claim that
\begin{equation}
( f_{\tau_k} \ln{f_{\tau_k}} , g_{\tau_k} \ln{g_{\tau_k}} ) \longrightarrow (f\ln{f} , g\ln{g}) \quad \text{ in } \quad L^1((0,T)\times\R)\,, \quad T>0\,. \label{be9}
\end{equation}
Indeed, by \eqref{be1} and the continuity of $r\mapsto r\ln{r}$ in $[0,\infty)$, we have for $T>0$
\begin{equation}
( f_{\tau_k} \ln{f_{\tau_k}} , g_{\tau_k} \ln{g_{\tau_k}} ) \longrightarrow (f\ln{f} , g\ln{g}) \quad \text{  a.e. in $(0,T)\times\R$}. \label{be10}
\end{equation}
 Moreover, it readily follows from \eqref{e3} (with $s=0$) that
\begin{equation}
( f_{\tau_k} \ln{f_{\tau_k}} , g_{\tau_k} \ln{g_{\tau_k}} )_{k\ge 1} \;\;\mbox{ is uniformly integrable in }\;\; L^1((0,T)\times\R;\R^2)\,, \label{be11}
\end{equation}  
while \eqref{e3}, \eqref{e4}, and the inequality $|r\ln{r}| \le 2 \sqrt{r}\ \max{\{r,1\}}$, $r\ge 0$, guarantee that, for $R>1$,
\begin{eqnarray}
\int_0^T \int_{\{|x|\ge R\}} |f_{\tau_k} \ln{f_{\tau_k}}|\, dxdt  & \le & 2\ \int_0^T \int_{\{|x|\ge R\}} \sqrt{f_{\tau_k}}\ \mathbf{1}_{[0,1]}(f_{\tau_k})\, dxdt \nonumber\\
& + & 2\ \int_0^T \int_{\{|x|\ge R\}} f_{\tau_k}^{3/2}\ \mathbf{1}_{(1,R)}(f_{\tau_k})\, dxdt \nonumber\\
& + & 2\ \int_0^T \int_{\{|x|\ge R\}} f_{\tau_k}^{3/2}\ \mathbf{1}_{[R,\infty)}(f_{\tau_k})\, dxdt \nonumber\\
& \le &  2\ \left( \int_0^T \int_{\{|x|\ge R\}} x^2\ f_{\tau_k}\, dxdt \right)^{1/2}\ \left( \int_0^T \int_{\{|x|\ge R\}} \frac{dxdt}{x^2}  \right)^{1/2} \nonumber\\ 
& + & 2\sqrt{R}\ \int_0^T \int_{\{|x|\ge R\}} f_{\tau_k}\, dxdt + \frac{2}{\sqrt{R}}\ \int_0^T \int_{\{|x|\ge R\}} f_{\tau_k}^2\, dxdt \nonumber\\ 
& \le & C\ \frac{1+T}{\sqrt{R}} + \frac{2}{R^{3/2}}\ \int_0^T \int_{\{|x|\ge R\}} x^2\ f_{\tau_k}\, dxdt + C\ \frac{1+T}{\sqrt{R}} \nonumber \\
& \le & C\ \frac{1+T}{\sqrt{R}}\,. \label{be12}
\end{eqnarray}
Due to \eqref{be10}-\eqref{be12}, we are in a position to apply Vitali's convergence theorem  (see, e.g., \cite[Theorem~2.24]{FL07} or \cite[Th\'eor\`eme~I.4.13]{Ka93}) and deduce the claim \eqref{be9} for $(f_{\tau_k})_{k\ge 1}$, the proof for $(g_{\tau_k})_{k\ge 1}$ being identical. Consequently, after possibly extracting a subsequence, we have also
\begin{equation}
(H(f_{\tau_k}(t)),H(g_{\tau_k}(t))) \longrightarrow (H(f(t)),H(g(t))) \quad \text{ a.e. in } \quad (0,\infty)\,, \label{be13}
\end{equation}
the functional $H$ being defined in \eqref{eq:H}. We next infer from \eqref{be2} and the Fatou lemma that, for $t>0$, 
\begin{eqnarray}
& & \int_0^t \left( \|\p_x f(s)\|_2^2 + R\ \|\p_x (f+g)(s)\|_2^2 \right)\, ds = \lim_{\delta\to 0} \int_\delta^t \left( \|\p_x f(s)\|_2^2 + R\ \|\p_x (f+g)(s)\|_2^2 \right)\, ds \nonumber\\
& &\le  \lim_{\delta\to 0}\ \liminf_{k\to\infty} \int_\delta^t \left( \|\p_x f_{\tau_k}(s)\|_2^2 + R\ \|\p_x (f_{\tau_k}+g_{\tau_k})(s)\|_2^2 \right)\, ds\,. \label{be8}
\end{eqnarray}
Owing to \eqref{be13} and \eqref{be8}, we may pass to the limit as $k\to\infty$ in \eqref{Hes} to obtain the assertion $(a)$ of Theorem~\ref{MT:1}, which completes its proof.

%%%%%%%%%%%%%%%%%%%%%%%%%%%%%%%%%%%%%%%%%%%%%%%%%%%%%%%%%%%%%%%%%%%%%%%%%%%
\section*{Acknowledgements} 
%%%%%%%%%%%%%%%%%%%%%%%%%%%%%%%%%%%%%%%%%%%%%%%%%%%%%%%%%%%%%%%%%%%%%%%%%%%

PhL warmly thanks Adrien Blanchet and Giuseppe Savar\'e for helpful discussions on the optimal transport approach to partial differential equations. This work was initiated during a visit of BVM at the Institut de Math{\'e}matiques de Toulouse, Universit\'e Paul Sabatier - Toulouse~3. He is grateful for the hospitality.

%%%%%%%%%%%%%%%%%%%%%%%%%%%%%%%%%%%%%%%%%%%%%%%%%%%%%%%%%%%%%%%%%%%%%%%%%%%
%%%%%%%%%%%%%%%%%%%%%%%%%%%%%%%%%%%%%%%%%%%%%%%%%%%%%%%%%%%%%%%%%%%%%%%%%%%
\appendix
\section{Some technical results}
%%%%%%%%%%%%%%%%%%%%%%%%%%%%%%%%%%%%%%%%%%%%%%%%%%%%%%%%%%%%%%%%%%%%%%%%%%%
%%%%%%%%%%%%%%%%%%%%%%%%%%%%%%%%%%%%%%%%%%%%%%%%%%%%%%%%%%%%%%%%%%%%%%%%%%%

We first collect some well-known properties of the functional $H$ defined by \eqref{eq:H}.

%%%%%%%%%%%%%%%%%%%%%%%%%%%%%%%%%%%%%%%%%%%%%%%%%%%%%%%%%%%%%%%%%%%%%%%%%%%
\begin{lemma}\label{le:ap1}
Let $h$ be a nonnegative function in $L^1(\R,(1+x^2) dx)\cap L^2(\R)$. Then $h\ln{h}\in L^1(\R)$ and there is a positive constant $C_\ell$ such that
\begin{eqnarray}
\int_\R h(x)\ |\ln{h(x)}|\, dx & \le & C_\ell + \int_\R h(x)\ \left( 1 + x^2 \right)\ dx + \|h\|_2^2\,, \label{eq:ap1} \\
H(h) & \ge & - C_\ell - \int_\R h(x)\ \left( 1 + x^2 \right)\ dx\,. \label{eq:ap2}
\end{eqnarray}
\end{lemma}
%%%%%%%%%%%%%%%%%%%%%%%%%%%%%%%%%%%%%%%%%%%%%%%%%%%%%%%%%%%%%%%%%%%%%%%%%%%

\begin{proof}
Introducing the function $\omega(x):= e^{-(1+x^2)}$, $x\in\R$,  and using the monotonicity of $r\mapsto r|\ln r|$ in $[0,1/e],$ we have
\begin{eqnarray*}
\int_\R h(x) |\ln{h(x)}|\, dx & = & \int_{\{h(x)<\omega(x)\}} h(x) |\ln{h(x)}|\, dx + \int_{\{\omega(x) \le h(x) \le 1\}} h(x) |\ln{h(x)}|\, dx \\
& & + \int_{\{h(x)>1\}}h(x) |\ln{h(x)}|\, dx \\
& \le & \int_{\{h(x)<\omega(x)\}} e^{-(1+|x|^2)}\ (1+x^2)\, dx + \int_{\{\omega(x) \le h(x) \le 1\}} h(x) (1+x^2)\, dx \\
&& +  \int_{\{h(x)>1\}} h(x) (h(x)-1)\, dx \\ 
& \le & \int_\R e^{-(1+|x|^2)}\ (1+x^2)\, dx + \int_\R h(x) (1+x^2)\, dx + \|h\|_2^2\,,
\end{eqnarray*}
whence \eqref{eq:ap1}. Similarly,
\begin{eqnarray*}
H(h) & \ge & \int_{\{h(x)<\omega(x)\}} h(x) \ln{h(x)}\, dx + \int_{\{\omega(x) \le h(x) \le 1\}} h(x) \ln{h(x)}\, dx \\
& \ge & - \int_{\{h(x)<\omega(x)\}} e^{-(1+|x|^2)}\ (1+x^2)\, dx - \int_{\{\omega(x) \le h(x) \le 1\}} h(x) (1+x^2)\, dx\,,
\end{eqnarray*}
from which \eqref{eq:ap2} readily follows.
\end{proof}

The next results  allowed us to identify the limit of some terms arising in  the derivation of the Euler-Lagrange equation in Lemma~\ref{le:sup1}.

%%%%%%%%%%%%%%%%%%%%%%%%%%%%%%%%%%%%%%%%%%%%%%%%%%%%%%%%%%%%%%%%%%%%%%%%%%%
\begin{lemma}\label{L:1}
Consider $h\in H^1(\R)$ and $\zeta\in C_0^\infty(\R)$. Setting $\zeta_\e := \id + \e\ \zeta$ for $\e>0$, we have 
\begin{equation}\label{1}
h\circ \zeta_\e \mathop{\longrightarrow}_{\e\to 0} h \quad\text{in}\quad L^2(\R) \qquad\text{and}\qquad \frac{h\circ \zeta_{\e} -h}{\e} \mathop{\rightharpoonup}_{\e\to 0} \zeta \p_x h \quad \text{in}\quad L^2(\R). 
\end{equation}
\end{lemma}
%%%%%%%%%%%%%%%%%%%%%%%%%%%%%%%%%%%%%%%%%%%%%%%%%%%%%%%%%%%%%%%%%%%%%%%%%%%

\begin{proof}
 Since $\zeta_\e$ is a $C^\infty-$diffeomorphism from $\R$ onto $\R$ for $\e$ small enough, its inverse $\zeta_\e^{-1}$ is well-defined and satisfies
\begin{equation}
\left| x - \zeta_\e^{-1}(x) \right| \le \e\ \|\zeta\|_\infty\,, \qquad x\in\R. \label{eq:ap3}
\end{equation}
It follows from the Cauchy-Schwarz inequality, the Fubini theorem, and \eqref{eq:ap3} that
\begin{eqnarray*}
\|h\circ \zeta_\e - h \|_2^2 & = & \int_\R \left( \int_x^{\zeta_\e(x)} \p_x h(y)\, dy \right)^2\, dx \le \int_\R \left| x-\zeta_\e(x) \right|\ \left| \int_x^{\zeta_\e(x)} |\p_x h(y)|^2\, dy \right|\, dx \\ 
& \le & \e\ \|\zeta\|_\infty \int_\R |\p_x h(y)|^2\ \left| y - \zeta_\e^{-1}(y) \right|\, dy \le \e^2\ \|\zeta\|_\infty^2\ \|\p_x h\|_2^2\,,
\end{eqnarray*}
which gives the first assertion in \eqref{1} and the boundedness of $\left( (h\circ \zeta_\e-h)/\e \right)_\e$ in $L^2(\R)$. Next, since $h\in H^1(\R)$, almost every $x\in\R$ is a Lebesgue point for $\p_x h$ and, for such points,
$$
\frac{h(x+\e\zeta(x))-h(x)}{\e} = \frac{1}{\e\zeta(x)}\ \int_x^{x+\e\zeta(x)} \p_x h(y)\, dy\ \zeta(x) \mathop{\longrightarrow}_{\e\to 0} \p_x h(x)\ \zeta(x)\,.
$$
Therefore, $\left( (h\circ \zeta_{\e} -h)/\e \right)_\e$ converges a.e. to $\zeta \p_x h$ as $\e\to 0$ and is bounded in $L^2(\R)$, and the second assertion in \eqref{1} readily follows from these two facts. 
\end{proof}

 The first assertion of Lemma~\ref{L:1} is actually true in a more general setting:
%%%%%%%%%%%%%%%%%%%%%%%%%%%%%%%%%%%%%%%%%%%%%%%%%%%%%%%%%%%%%%%%%%%%%%%%%%%
\begin{lemma}\label{le:ap2}
Consider $h\in H^1(\R)$ and a sequence $(\zeta_\e)_{\e>0}$ of functions in $C_0^\infty(\R)$ such that $\omega_\e:= \|\zeta_\e- \id\|_\infty\longrightarrow 0$ as $\e\to 0$. Then
$$
h\circ \zeta_\e \mathop{\longrightarrow}_{\e\to 0} h \quad\text{in}\quad L^2(\R)\,. 
$$
\end{lemma}
%%%%%%%%%%%%%%%%%%%%%%%%%%%%%%%%%%%%%%%%%%%%%%%%%%%%%%%%%%%%%%%%%%%%%%%%%%%

\begin{proof}
As in the proof of Lemma~\ref{L:1}, it follows from the Cauchy-Schwarz inequality and Fubini's theorem that
\begin{eqnarray*}
\|h\circ \zeta_\e - h \|_2^2 & \le & \int_\R \left| x-\zeta_\e(x) \right|\ \left| \int_x^{\zeta_\e(x)} |\p_x h(y)|^2\, dy \right|\, dx  \le  \omega_\e\ \int_\R \int_{x-\omega_\e}^{x+\omega_\e} |\p_x h(y)|^2\, dydx \\
&\le& 2\ \omega_\e^2\ \|\p_x h\|_2^2\,,
\end{eqnarray*}
and the right-hand side of the above inequality converges to zero as $\e\to 0.$
\end{proof}

%%%%%%%%%%%%%%%%%%%%%%%%%%%%%%%%%%%%%%%%%%%%%%%%%%%%%%%%%%%%%%%%%%%%%%%%%%%
\bibliographystyle{abbrv}
\bibliography{BMPhL}
%%%%%%%%%%%%%%%%%%%%%%%%%%%%%%%%%%%%%%%%%%%%%%%%%%%%%%%%%%%%%%%%%%%%%%%%%%%

\end{document}